\documentclass[12pt, a4paper]{amsart}
\usepackage{amsmath,amssymb,amscd,amsfonts}

\usepackage{ifthen}
\usepackage[T1]{fontenc}
\usepackage[utf8]{inputenc}
\usepackage[all]{xy}
\usepackage{graphicx}
\usepackage{enumerate}
\usepackage{xspace}
\usepackage{pdfsync}
\usepackage{epic}
\usepackage{dsfont}

\newtheorem{theorem}{Theorem}[section]
\newtheorem{remark}[theorem]{Remark}
\newtheorem{definition}[theorem]{Definition}
\newtheorem{lemma}[theorem]{Lemma}

\newtheorem{proposition}[theorem]{Proposition}
\newtheorem{conjecture}[theorem]{Conjecture}
\newtheorem{corollary}[theorem]{Corollary}
\newtheorem{example}[theorem]{Example}
\newtheorem{question}[theorem]{Question}

\newcommand\C{{\mathbb{C}}}

\def\Sym{\mathop{\rm Sym}\nolimits}

\def\cE{{\mathcal E}}

\def\cG{{\mathcal G}}
\def\cH{{\mathcal H}}
\def\cF{{\mathcal F}}

\begin{document}
\title[]{Numerical and Kodaira dimensions of cotangent bundles}

\author{Fr\'ed\'eric Campana}
\address{Universit\'e Lorraine \\
 Institut Elie Cartan\\
Nancy \\ }

\email{frederic.campana@univ-lorraine.fr}



\date{\today}

\begin{abstract} We conjecture the equality of the numerical and Kodaira dimensions $\nu_1^*(X)$ and $\kappa_1^*(X)$ for the cotangent bundle of compact K\"ahler manifolds $X$, generalising the classical case of the canonical bundle. We show or reduce it to the classical case of the canonical bundle for some peculiar manifolds: among them, the rationally connected ones, or resolutions of varieties with klt singularities and trivial first Chern class, in which case we show that $\nu_1^*(X)=\kappa_1^*(X)=q'(X)-dim(X)$, where $q'(X)$ is the maximal irregularity of a finite \'etale cover of $X$. The proof rests on the Beauville-Bogomolov decomposition, and a direct computation for smooth models of quotients $A/G$ of complex tori by finite groups. We conjecture that these equalities hold true, much more generally, when $X$ is `special'. The invariant $\kappa_1^*$ was already introduced and studied by Fumio Sakai in \cite{Sak}, the particular case of the preceding conjecture when $\kappa_1^*(X)=-dim(X)$ was introduced and studied in \cite{HP20}. 

\

R\'ESUM\'E:Nous conjecturons l'\'egalit\'e entre les dimensions num\'erique et de Kodaira $\nu_1^*(X)$ et $\kappa_1^*(X)$ pour le fibr\'e cotangent des vari\'et\'es K\"ahl\'eriennes compactes $X$, g\'en\'eralisant le cas classique du fibr\'e canonique.Nous la d\'emontrons ou la r\'eduisons au cas classique du fibr\'e canonique pour certaines classes de vari\'et\'es, parmi lesquelles: les vari\'et\'es rationnellement connexes, ainsi que les mod\`eles lisses des vari\'et\'es \`a singularit\'es klt et  premi\`ere classe de chern triviale, pour lesquelles nous montrons que $\nu_1^*(X)=\kappa_1^*(X)=q'(X)-dim(X)$, o\`u $q'(X)$ est l'irregularit\'e maximale des rev\^etements \'etales de $X$. La preuve repose sur la d\'ecomposition de Bogomolov-Beauville, et un calcul direct pour les mod\`eles lisses des quotients $A/G$ de tores complexes par un groupe fini. Nous conjecturons que ces \'egalit\'es restent vraies, bien plus g\'en\'eralement, lorsque $X$ est `sp\'eciale'. L'invariant $\kappa_1^*$ a \'et\'e introduit et \'etudi\'e auparavant par Fumio Sakai dans \cite{Sak}, et l'\'egalit\'e  $\nu_1^*(X)$ and $\kappa_1^*(X)$ conjectur\'ee et \'etudi\'ee lorsque $\kappa_1^*(X)=-dim(X)$ dans \cite{HP20}. 

\end{abstract}

\maketitle

\tableofcontents

\section{Introduction: Cotangent dimensions.}

Let $X$ be a connected compact complex manifold of dimension $n$, and $L$ a line bundle on $X$. Two invariants of $L$, its numerical and  Kodaira-Iitaka dimensions, are classically defined\footnote{$\nu(X,L):=min\{k\in \Bbb Z\vert \forall D,\exists C(D)>0\vert h^0(X,mL+D)\leq C(D).m^k$, $m \to +\infty\}$, $\forall D\geq 0$; $\kappa$ is defined similarly with $D=0$.}: $$\nu(X,L)\geq \kappa(X,L)\in \{-\infty,0,\dots,n\}.$$ Both are preserved by surjective pullbacks. We consider here the case when $X$ is projective, or compact K\"ahler. The equivalent invariants $\nu^*(X,L),\kappa^*(X,L)$ are also sometimes used: they coincide with $\nu(X,L),\kappa(X,L)$, except when these take the value $-\infty$, in which case their assigned value is $-1$.

\begin{remark}\label{rnu=k}
In general, we have $\kappa(X,L)<\nu(X,L)$ (see Example \ref{nu>k}). There are however cases when we have equality:

1.If $\nu(X,L)=n$, then $\kappa(X,L)=n$ (we skip the easy proof). 

Consequently:

2. If $\kappa(X,L)=(n-1)$, then $\nu(X,L)=(n-1)$.
\end{remark}

\smallskip

The most important invariants of bimeromorphic geometry are the canonical bundle $L:=K_X$, and $\kappa(X):=\kappa(X,K_X), \nu(X):=\nu(X,K_X)$. 

\smallskip

A central conjecture, with far-reaching consequences, is: 

\begin{conjecture}\label{ConjK} If $X$ is compact K\"ahler, one has: $\kappa(X)=\nu(X)$.
 \end{conjecture}

If $E$ is a vector bundle of rank $r>0$ on $X$, let $\pi:P_E:=\Bbb P(E)\to X$ be the bundle of hyperplanes of $E$ over $X$ and $L_E:=\mathcal{O}_{P_E}(1)$.

\begin{definition} In this situation, define\footnote{The invariant $\lambda(X,E)$ defined in \cite{Sak} thus coincides with $\kappa^*(X,E)$.}: 

$\nu(X,E):=\nu(P_E,L_E)-(r-1)$ and $\kappa(X,E):=\kappa(P_E,L_E)-(r-1)$. Both take values in $\{-\infty, -(r-1),\dots,n\}$.

We shall also use the equivalent invariants:

 $\nu^*(X,E):=\nu^*(P_E,L_E)-(r-1)$ , $\kappa^*(X,E):=\kappa^*(P_E,L_E)-(r-1)$, which thus coincide with $\nu(X,E), \kappa(X,E)$ respectively, except that the value $-\infty$ is replaced by $-r$. They take values in $\{-r,-(r_1),\dots,n\}$.
 
  These starred invariants behave better under products.

We say that $E$ is not pseudoeffective if $L_E$ is not pseudoeffective, that is: if $\nu(X,E)=-\infty$.
\end{definition}

Substracting $(r-1)=dim(P_E)-dim(X)$ is a natural normalisation. Indeed:  it does not change anything when $r=1$, $\kappa(X,T\otimes L)=\kappa(X,L)$ and $\nu(X,T\otimes L)=\nu(X,L)$ when $T$ is a trivial bundle, $L$ is any line bundle. Finally: $\kappa(X,E)=n$ if and only if $E$ is big (ie: if so is $L_E$).

\begin{example}\label{nu>k} In general, $\nu(E)\geq \kappa(E)$, the inequality can be strict. Let, for example, $X$ be complex projective of dimension $n>0$, and let $0\to \mathcal{O}_X\to E\to \mathcal{O}_X\to 0$ be a non-split extension. Then $\kappa(X,E)=-1$, and $\nu(X,E)=0$ (proof in \ref{exnu>k}).\end{example}

\begin{remark}\label{rnu=k'} From Remark \ref{rnu=k} we deduce that $\nu(X,E)=\kappa(X,E)$ if either $\nu(X,E)=n$, or if $\kappa(X,E)=(n-1)$. 
\end{remark}

\smallskip

$\bullet$  The cases $E=\Omega^p_X,p>0,$ are of central interest for the bimeromorphic classification. We write $\nu_p(X):=\nu(X,\Omega^p_X)$, $\kappa_p(X):=\kappa(X,\Omega^p_X)$, and so: $\kappa_n(X)=\kappa(X)$. We shall also use the equivalent variants $\nu_p^*(X):=\nu^*(X,\Omega^p_X)$ and $\kappa_p^*(X):=\kappa^*(X,\Omega^p_X)$. 

These invariants are preserved by birational equivalence, finite \'etale covers, and increase under dominant rational maps. These invariants are thus still defined if $X$ is singular irreducible, by taking any smooth model of $X$.

\smallskip

We formulate the extension to every $p>0$ of Conjecture \ref{ConjK}:

\begin{conjecture}\label{ConjO} We have: $\nu_p(X)=\kappa_p(X)$ for any compact connected K\"ahler manifold $X$ and any $p>0$.
\end{conjecture}

See \S.\ref{motivations} below for some motivations for Conjecture \ref{ConjO}.

\begin{remark}\label{rnu=k"} From Remark \label{rnu=k'} we deduce that $\nu_1(X)=\kappa_1(X)$ if either $\nu_1(X)=n$, or if $\kappa_1(X)=(n-1)$.
\end{remark}

There are two natural generalisations of Conjecture \ref{ConjO} that should also be considered:

1. The case of smooth orbifold pairs $(X,D)$, needed to deal with the general classical case where $D=0$ (see Remark \ref{Delta} and \S\ref{Sspec}).

2. The general $T$- symmetrical tensors associated to Young's tableaux studied in \cite{BR}, which include both symmetric and alternating forms, and jet differentials, for their relevance to hyperbolicity problems.

\smallskip

The invariant $\kappa_1^*(X)$ has already been introduced and studied in \cite{Sak}\footnote{Where it is denoted $\lambda(X)$.}, several ideas of which are used below. When $\kappa_1(X)=-\infty$, the preceding conjecture \ref{ConjO} coincides with the one independently\footnote{We formulated Conjecture \ref{ConjO} in Spring 2017 independently of \cite{HP20}, but did not succeed in proving it for non isotrivial elliptic surfaces with $\kappa=1$ by an algebro-geometric approach, and let it thus unpublished. The proof given in \cite{HP20} indeed rests on analytic arguments.} formulated in \cite{HP20}. The present text has important overlap with \cite{HP20}, especially in \S\ref{Ssurf}, but its scope seems to be different.

The invariants $\kappa_p$ and $\nu_p$ for the cotangent bundles are easily seen (see \cite{Sak}, Theorem 1) to be preserved by finite \'etale covers and birational equivalence. For a product $X=Y\times Z$, and $p=1$, we have\footnote{For $p>1$, the expression for $\kappa^*_p(Y\times Z)$ is not so simple. For example, if $S,T$ are surfaces, $T$ an Abelian surface, the computation of $\kappa_2(S\times T)$ involves an estimate of $h^0(S\times T, Sym^m(\Omega^1_S\otimes \Omega^1_T))=h^0(S, \oplus_{j=0}^{j=m}Sym^j(\Omega^1_S)\otimes Sym^{m-j}(\Omega^1_S))$. When $S$ is of general type, this estimate can be obtained by the method used in \cite{GG}, Proposition 1.10, and statement 1.21. Moreover, in general $Sym^m(A\otimes B)$ is not expressible in terms of direct sums of tensor products of $Sym^{\bullet}A\otimes Sym^{\bullet}B$.}: $\kappa_1^*(Y\times Z)=\kappa_1^*(Y)+\kappa_1^*(Z)$. Also, if $f:X\to Y$ is onto, then $\nu_1(X)\geq \nu_1(Y)$, and $\kappa_1(X)\geq \kappa_1(Y)$ . 

Depending on the property considered, either $\nu^*,\kappa^*$ or $\nu,\kappa$ are better suited for a simple formulation, and we shall use the one better suited to the cases at hand.

When $X$ is of general type, no special pattern seems to emerge to relate $\kappa_1(X)$ to other algebro-geometric invariants of $X$, except possibly the fundamental group, by \cite{BKT}: the linear representations of $\pi_1(X)$ have finite image if $\kappa_1(X)=-\infty$. Conjecture \ref{ConjO} seems particularly difficult in this case. However, for submanifolds with ample normal bundle in Abelian varieties, and complete intersections in projective spaces, the value of $\kappa_1$ has been determined in \cite{Sak}, Theorems 7 and 8. His method adapts to the determination of $\nu_1$, and thus to solve Conjecture \ref{ConjO} in these cases. 

We deal here with mainly with the `opposite' case of `special' manifolds, and solve Conjecture \ref{ConjO}  only for particular classes of them.

In contrast to $\kappa(X)=\kappa_n(X), \kappa_1(X)$ and $\nu_1(X)$ are not deformation, and so not topological invariants. This already happens for (`non-special') elliptic surfaces with $\kappa=1$ (see \cite{Sak}, Example 3). We expect however these two invariants to be determined by the fundamental group of $X$ whenever $X$ is `special' (see Conjecture \ref{Conjspec}). We show this in the particular case where $X$ is a smooth model of a compact K\"ahler variety with klt singularities and $c_1=0$ (see Theorem \ref{c_1=0}). The proof rests on the singular Bogomolov-Beauville decomposition and the elementary but lengthy case of quotients of compact tori by finite groups actions (which is the main contribution of this article).

Notice also that Conjectures \ref{ConjK} and \ref{ConjO} do not  extend to arbitrary subsheaves of the cotangent bundle (although the known situations where they fail appear to be rather exceptional). See Remark \ref{motivations}.

The present text proves (or reduces to Conjecture \ref{ConjK}) the above conjecture \ref{ConjO} in the following peculiar situations:

\smallskip

{\bf 1. Rationally connected manifolds.} (\S\ref{RC})

Here we show, assuming Conjecture \ref{ConjK}, that $X$ is rationally connected if (and only if) $\kappa_p(X)=-\infty, \forall p>0$. In other words: Conjecture \ref{ConjO} follows from Conjecture \ref{ConjK} when $\kappa_p(X)=-\infty, \forall p>0$. We also show a relative version (see Theorem \ref{tRQ}) according to which Conjecture \ref{ConjO} holds for any $X$ if it holds when $X$ is not uniruled, or equivalently (still assuming Conjecture \ref{ConjK}), when $\kappa(X)\geq 0$. 

We finally stress some similarities between (complex projective) manifolds $X$ with $\kappa_1(X)=-\infty$ and Rationally connected manifolds, and give our initial motivations for Conjecture \ref{ConjO}.

\smallskip

{\bf 2. Submanifolds of Abelian varieties and projective spaces.}

We recall the value of $\kappa_1$ determined in \cite{Sak}, Theorems 7 (resp. Theorem 8), when $X$ is a submanifold with ample normal bundle in a Abelian variety (resp. when $X$ is a complete intersection of codimension smaller than its dimension in a projective space), and show that his method of proof permits to compute $\nu_1=\kappa_1$ as well in these cases. We conclude this section with remarks on related situations.

{\bf 3. Klt varieties with $c_1=0$ and their smooth models.}(\S\ref{Sc_1=0}, \S\ref{Storusq})

We prove that, for any $X$ admitting\footnote{Conjecturally, this is true whenever $\kappa(X)=0$.} a birational model $X'$ with klt singularities and $c_1=0$, we have: $\nu_1^*(X)=\kappa_1^*(X)=q'(X)-dim(X)$. Here $q'(X)$ is the maximum of the irregularities of the finite {\bf \'etale} covers of $X'$, or equivalently the maximal rank of the Abelianisations of the finite index subgroups of $\pi_1(X')$. The invariant $q'(X)$ should be carefully distinguished from its {\bf quasi-\'etale} variant $q^+(X')$, which is not a deformation invariant of the smooth models. In \S\ref{Sc_1=0}, using the singular Bogomolov-Beauville decomposition, we reduce the computation of $\nu_1^*(X)$ and $\kappa_1^*(X)$ to the particular case where  $X'=A/G$ where $A$ is a compact complex torus, and $G$ a finite group of automorphisms of $A$. In \S\ref{Storusq}, which is the main contribution of the present text, we then show the equalities above for any  $X'$ of the form $A/G$.

Smooth models of klt varieties with $c_1=0$, and especially of quotients $A/G$, build a small, but possibly representative class of `special' manifolds considered in the next section. 

\smallskip

{\bf 4. Special manifolds.}(\S\ref{Sspec})

Recall that a compact K\"ahler manifold $X$ is said to be `special' if $\kappa(X,L)<p$, for any $p>0$, and any rank-one coherent subsheaf $L\subset \Omega^p_X$. 

If $X$ has a dominant fibration $f:X\dasharrow Z$, with $\kappa(Z)=p=dim(Z)$, it is not special. In particular, $X$ is not special if it is of general type. The specialness condition is however considerably stronger than the absence of fibrations with base of general type.

Basic examples of special manifolds are rationally connected manifolds, and manifolds with $\kappa=0$. Conditionally in the Conjecture $C_{n,m}^{orb}$ (stated in \cite{Ca04}, Conjecture 4.1), special manifolds are towers of fibrations with fibres either rationally connected or with $\kappa=0$ in a suitable `orbifold' sense.

We conjecture that the property shown for smooth models of klt varieties with $c_1=0$ holds true for all special manifolds.

\begin{conjecture} \label{Conjspec} If $X$ is special, $\nu_1^*(X)=\kappa_1^*(X)=q'(X)-dim(X)$.

In particular: $\nu_1(X)=-\infty$ if and only  $q'(X)=0$.\end{conjecture}

It is easy to see that $\kappa_1^*(X)\geq q'(X)-dim(X)$ if $X$ is special, so the content of the conjecture is the reverse inequality, saying that symmetric differentials come from the Albanese variety, after a suitable finite \'etale cover of $X$.

We explain in this same section how one could expect  to reduce Conjecture \ref{Conjspec} to statements about the behaviour of $\kappa_1, \nu_1$ on fibrations with generic fibres having either: 

1. $\kappa^+=-\infty$, or:

2.  $\kappa=0$ and $q'=0$, or: 

3. being Abelian varieties,

 One however then needs to consider the much more delicate `orbifold' context.

We also proposed in \cite{CRM}, Remark 7.3,  the following conjecture relating numerical and Kodaira dimensions:

\begin{conjecture}\label{Conjnuspec} If $X$ is special, then $\nu(X,L)<p,\forall p>0,\forall L\subset \Omega^p_X$, of rank one.
\end{conjecture}

This is proved for $p=1$ in \cite{PRT}. A subbundle $L\subset \Omega^1_X$ with $\nu(X,L)=1$ does not need, however, to have $\kappa(X,L)=1$, as the examples of Brunella show (\cite{Brun}, Exemples 1 and 2, p.132). See Remark \ref{motivations}.(2) and question \ref{rnu>kappa}.

\smallskip

{\bf 5. Surfaces not of general type.} (\S\ref{Ssurf})

For surfaces of general type, no general pattern seems to emerge to relate $\kappa_1$ or $\nu_1$ to other algebro-geometric invariants if $c_1^2\leq c_2$ on minimal models. With the exception of $\pi_1$ whose linear representations turn out to have finite images if $\kappa_1=-\infty$ (\cite{BKT}). Even the border case $c_1^2=c_2$ is open. If $3c_1<c_2$, the classification established in \cite{R} implies that $\pi_1^{alg}$ is finite.

\begin{proposition} \label{psurf} If $X$ is a smooth projective surface not of general type, then $\nu_1^*(X)=\kappa_1^*(X)$, unless possibly when $\kappa(X)=1$, the elliptic fibration $f:X\to C$ is not isotrivial, and its orbifold base has a pseudoeffective canonical bundle.\end{proposition} 

This was essentially proved in \cite{HP20} when $\kappa_1(X)=-\infty$.

\medskip

The proof rests on classification.  The computation of $\kappa_1$ was essentially done in \cite{Sak}. When $\kappa=1$, and the elliptic fibration is isotrivial, we give a shorter proof of this equality.

Two cases are left open: non-isotrivial elliptic surfaces with $\kappa=1$ over a curve of genus $g$ either $1$, or at least $2$. A possible approach might be using the solution of Conjecture \ref{ConjK} on the $3$-fold pair $(\Bbb P(\Omega^1_X), D)$ for a suitable effective $\Bbb Q$-divisor linearly equivalent to $\mathcal{O}_{\Bbb P(\Omega^1_X)}(2)$.

\medskip

We raise finally in {\bf \S\ref{Sfol}} a question motivated by the fact that Conjecture \ref{ConjO} only partially extends to saturated subsheaves of $\Omega^p_X$.

\subsection{Motivations for Conjecture \ref{ConjO}} \label{motivations}

\

 1. The initial reason here for considering manifolds with $\nu_1(X)=-\infty$ comes fom the algebraicity criterion for foliations proved (even if not explicitely stated there) in \cite{CP19}, 
 according to which a foliation $\cF\subset T_X$ on $X$ projective has algebraic leaves if its dual $\cF^*$ is not pseudo-effective.  The examples of Brunella (\cite{Brun}, Ex.1, p.132) of rank-one regular foliations $\cF$ with leaves non-algebraic such that $\nu(\cF^*)=1$ and $\kappa(\cF^*)=-\infty$ on irreducible quotients of the bidisc show that the condition $\kappa(\cF^*)=-\infty$ does not imply the algebraicity of leaves. 
 
 Although the algebraicity criterion is void if $\cF=T_X$, 
 it suggests to consider manifolds $X$ with $\Omega^1_X$ not pseudo-effective, that is, with $\nu_1(X)=-\infty$. The link with rationally connected manifolds also stems from \cite{CP19}, 
 where the case in which the negativity of the minimal slope of $\cF$ with respect to some movable class on $X$ (which strengthens the non-pseudo-effectivity of $\cF^*)$ is shown to imply the rational connectedness of the leaves of $\cF$. 
  It is however not immediately clear whether or not the leaves should have a non-pseudo-effective cotangent bundle when $\cF^*$ is not pseudo-effective.

  2. Projective (or compact K\"ahler) manifolds $X$ with $\nu_1(X)=-\infty$ appear to bear some similarities with the
 rationally connected ones, defined by $\nu_p(X)=-\infty,\forall p>0$. The latter are simply-connected, while the former 
 have finite fundamental groups, at least at the linear representation level, by \cite{BKT}. On the other hand, they may lie at the opposite 
 of the spectrum based on the classification according to the positivity of the canonical bundle: the former ones can have $K_X$ ample, while the latter ones may be Fano (with $-K_X$ ample). 
  
  3. Set, as in \cite{Ca95}, $\gamma d(X):=dim(X)-\pi(X)$, where $\pi(X)$ is the largest dimension of a connected submanifold $Z$ through a general point of $X$ such that the image of $\pi_1(Z)$ inside $\pi_1(X)$ is finite. 
  
  The comparison theorem $\kappa^+(X)\geq \gamma d(X)$ if $\chi(\mathcal{O}_X)\neq 0$ of \cite{Ca95} leads to bound from above $\kappa^+(X):=max\{\kappa(X,det(\cF)), \forall \cF\subset \Omega^p_X,\forall p>0\}$. In \cite{Ca95}, $\kappa^+(X)$ was conjecturally equal to $\kappa(X)$ if $\kappa(X)\geq 0$. When $K_X$ is pseudo-effective, the bound $\kappa^+(X)\leq \nu(K_X)$ is proved in \cite{CP19}. Both upper bounds should coincide according to Conjecture \ref{ConjK}. Again, the MRC fibration of $X$ permits to reduce to the case when $X$ is not uniruled, since $\gamma d(X)$ and $\kappa^+(X)$ are the same for $X$ and its `rational quotient'.

\section{Rationally connected manifolds.} \label{RC}

\subsection{Uniruledness and rational connectedness}

Recall first the standard observation:

\begin{proposition}\label{punr} Assume Conjecture \ref{ConjK}. The following are then equivalent if $X$ is projective.

1. $X$ is uniruled.

2. $\kappa(X)=-\infty$.
\end{proposition}

\begin{proof} Assume $\kappa(X)=-\infty$. By Conjecture \ref{ConjK}, $\nu(X)=-\infty$. Then $X$ is uniruled, by \cite{MiMo86} and \cite{BDPP}. More precisely: \cite{BDPP} shows that if $\nu(X)=-\infty$, $X$ is covered by projective curves with negative intersection with $K_X$, and \cite{MiMo86} then shows that $X$ is uniruled. The reverse implication is easy.\end{proof}

Let us next give a standard example in which Conjecture \ref{ConjO} follows from the classical Conjecture \ref{ConjK}:

\begin{proposition}\label{exRC} Assume Conjecture \ref{ConjK} to be true. If a compact K\"ahler manifold $X$ has $\kappa_p(X)=-\infty,\forall p>0$, then $X$ is rationally connected, and so $\nu_p(X)=-\infty,\forall p>0$. \end{proposition}

\begin{proof}Notice first that $X$ is then projective, by a theorem of Kodaira, since $h^0(X,\Omega^2_X)=0$. Let then $r_X:X\to R_X$ be the rational quotient (also called the MRC fibration) of $X$. If $p:=dim(R_X)>0$, we have: $-\infty=\kappa_p(X)+rank (\Omega^p_X)-1\geq \kappa_p(R_X)=\kappa(R_X)$ (the first equality is obvious, since sections of $m.K_{R_X}$ lift to sections of $Sym^m(\Omega^1_{X})$). By Proposition \ref{punr}, $R_X$ is uniruled, contradicting \cite{GHS} since $p>0$. Hence $p=0$, $X$ is rationally connected, and so $\nu_p(X)=-\infty,\forall p>0$.\end{proof}

\begin{remark} 1. The preceding proof rests in an essential way on positive characteristic methods, through the use of \cite{MiMo86}.

2. When $-K_X$ is ample (i.e: $X$ is Fano), $X$ is rationally connected, and so the equalities $\nu_p(X)=-\infty,\forall p>0$ are known unconconditionally, but using \cite{MiMo86}, again. It would be interesting to find a characteristic zero argument to deduce that $\nu_p(X)=-\infty,$ $ \forall p>0$, or even just that $\kappa_p(X)=-\infty,$ $ \forall p>0,$ when $-K_X$ is ample. 

3. Another partial result (\cite{BC}, see Remark \ref{rem ample}), proved unconditionally using $L^2$ methods, is that $\pi_1(X)=\{1\}$ if $\kappa_p(X)=-\infty, \forall p>0$, without proving first that $X$ is rationally connected. \end{remark}

\subsection{Relative version}

The following relative version of the preceeding result, applied to the rational quotient (or MRC), shows that Conjecture \ref{ConjO} is true for any variety $X$ if it is true  for non uniruled varieties.

\begin{theorem}\label{tRQ} Let $f:X\to B$ be a fibration with rationally connected fibres, and $X,B$ complex projective smooth of dimensions $n,b$. Then $\nu_p(X)+C_n^p=\nu_p(B)+C_b^p$ and $\kappa_p(X)+C_n^p=\kappa_p(B)+C_b^p, \forall p>0$.

More generally, we have:

$\nu(X, \otimes^m(\Omega^1_X))+n^m= \nu(B, \otimes^m(\Omega^1_B))+b^m,$ and:

 $\kappa(X, \otimes^m(\Omega^1_X))+n^m= \kappa(B, \otimes^m(\Omega^1_B))+b^m, \forall m>0$.
 
 The addition of $C_n^p,n^m$ is a consequence of the normalisation of $\nu,\kappa$. \end{theorem}

\begin{proof} We prove it for $\nu_1$, the proof for $\kappa_1$ being obtained by taking $D=0$ in the next proof. 
The inequality $\nu_1(B)\leq \nu_1(X)$ is clear, so we just need to prove the opposite inequality. Flattening $f$ by suitable modifications of $B$ and $X$, by means of Raynaud or Hironaka's theorems, we shall assume that $f$ is `neat', which means that the $f$-exceptional divisor $\cE$ is also $u$-exceptional for some birational morphism $u:X\to X_0$, with $X_0$ smooth. By Hartog's theorem, the sections of $Sym^m(\Omega^1_X)\otimes u^*(A_0)$ over $(X\setminus \cE)$ thus all extend to $X$, and we shall ignore their possible poles occuring over $\cE$, for any $m>0$, since those do not alter the spaces of sections, for any $A_0$ ample on $X_0$, which is sufficient to compute $\nu_1(X)$.

To simplify notations, we write $S^m(X):=Sym^m(\Omega^1_X)$, $S^m(X/B):=(Sym^m(\Omega^1_{X/B})/Torsion)^{**}$, and $\bullet^{sat}$ for the saturation inside $S^m(X)$ of a subsheaf $\bullet\subset S^m(X)$. 

Let $D\subset B$ be an irreducible divisor, and let $f^*(D):=\sum_kt_k.E_k+R$, where $R$ is $f$-exceptional, and the $E_k's$ are distinct prime divisors such that $f(E_k)=D,\forall k$, the $t_k's$ being positive integers. 
By \cite{GHS}, $inf_k(\{t_k\})=1$. Thus $E:=f^*(D)-\sum_kE_k-R$ is `partially supported on the fibres of $f$' (see Definition \ref{PSF} below).

At the generic point of any $E_k$, so in codimension at least $2$ on $X$, we can choose local coordinates $(x):=(x_1,\dots, x_n)$ such that $f((x))=(y_1:=x_1^{t_k}, y_2=x_2,\dots,y_d:=x_d)$ in suitable local coordinates $(y)=(y_1,\dots,y_d)$ of $B$, where $d:=dim(B)$. By a simple local computation in these coordinates, we thus have\footnote{A more precise version is obtained using Log-differentials, but the coarse one used here is sufficient for our purposes.}: $f^*(S^{m-j}(B))^{sat}\otimes S^j(X/B)\subset f^*(S^{m-j}(B)((m-j).(t_k-1).E_k))\otimes S^j(X/B)$, and $f^*(S^{m-j}(B))^{sat}= f^*(S^{m-j}(B))$ on the generic point of the (always existing, by \cite{GHS}) components $E_k$ with $t_k=1$.

We have, for any $m>0$, over the generic fibre of $f$, a natural filtration of $S^m(X)$ with 
quotients $f^*(S^{m-j}(B))\otimes S^j(X/B)$, for $0\leq j\leq m$. By the preceding observation, we thus get outside of $\cE$ and an additional codimension $2$ locus, a filtration of $S^m(X)$ with associated graded the reflexive sheaves $\oplus_{j=0}^{j=m}f^*(S^{m-j}(B))^{sat}\otimes S^j(X/B)$. This implies, for any ample $A$ on $X$, with $t\geq t_k,\forall k$, and $E$ partially supported on the fibres of $f$, that:
$$h^0(X,S^m(X)\otimes A)\leq \oplus_{j=0}^{j=m}h^0(X,f^*(S^{m-j}(B)((m-j).t)E)\otimes S^j(X/B)\otimes A)$$
Let $\cG_j:=S^j(X/B)\otimes A$. Since $\nu_1(X_b)=-\infty$ on the generic, smooth, fibre $X_b$ of $f$, there is an integer $m_A$, independent of $m$, for $A$ fixed, such that the torsion free sheaf $f_*(\cG_j)$ vanishes generically, hence everywhere, on $B$, for $j\geq m_A$. 

We thus have: $h^0(B,S^{m-j}(B)\otimes f_*(\cG_j))=0$ if $j\geq m_A$, and so:
$h^0(X,S^m(X)\otimes A)\leq \sum_{j=0}^{j=m_A}h^0(B,S^{m-j}(B)(s.E)\otimes \cG_j)$, where $s>0$ is defined in Lemma \ref{lemmapsf} (proved in the next subsection), and taken sufficiently large, so as to work simultaneously for all of the finitely many sheaves $\cG_j$, $j\leq m_A$.

We thus have: $$h^0(X,S^m(X)\otimes A)\leq \sum_{0\leq j<m_A}h^0(X,S^{m-j}(B)\otimes f_*(\cG_j (sE)).$$
Applying the next Lemma \ref{inject}, we obtain an injection $f_*(\cG_j (sE))\to A'^{\oplus R}$ for some ample $A'$ on $B$, and some positive integer $R$, this being valid for any $j\leq m_A$.

Thus $h^0(X,S^m(X)\otimes A)\leq R.(\sum_{0\leq j<m_A}h^0(X,S^{m-j}(B)\otimes A'))$.

Since $ h^0(B,S^{m-j}(B)\otimes A')\leq C.m^{\nu_1(B)+dim(B)-1}$, for some constant $C=C(A')$, we get:
$h^0(X,S^m(X)\otimes A)\leq C.R.m^{\nu_1(B)+dim(B)-1}$, which is the claimed inequality.

For $\otimes^m(\Omega^1_X)$, and so each $\nu_p$, the proof is the same, using filtration by ordered $m$-tuples. 
\end{proof}

We used the following lemma (and also Lemma \ref{lemmapsf}, treated in the next subsection):

\begin{lemma}\label{inject} Let $\cF$ be a torsionfree sheaf on $B$. There exists an integer $R>0$ and an ample line bundle 
$A$ on $B$ such that $\cF$ injects into $A^{\oplus R}$ 
\end{lemma} 
\begin{proof} We may assume $\cF$ to be reflexive since it injects into its bidual $\cF^{**}$. Choose $A$ ample so that $\cF^*(A)$ is generated by its global sections. 
Dualise the corresponding surjective map $\mathcal{O}_B^{\oplus R}\to \cF^*(A)$ to obtain the injective map $\cF^{**}(-A)\to \mathcal{O}_B^{\oplus R}$, and tensorise with $A$.
\end{proof}

The proof of Theorem \ref{tRQ} actually works in a broader context:

\begin{corollary}\label{cRQ} Let $f:X\to B$ be a fibration between complex projective manifolds. Assume that:

1. $f$ is `neat'.

2. $\nu_1(X_b)=-\infty$, for $X_b$ a generic smooth fibre of $f$.

3. For each prime divisor $D\subset B$, $f^*(D)$ contains at least one reduced component which is not $f$-exceptional.

Then $\nu_1(X)+dim(X)=\nu_1(B)+dim(B)$. 
\end{corollary} 

 The example of the next Remark \ref{Delta} shows that Condition $3$ in Corollary \ref{cRQ} is essential. With the notion of `orbifold base' $(B,\Delta_f)$ of a fibration $f:X\to B$ introduced in \cite{Ca04}, it means that $\Delta_f=0$, or equivalently, that $f$ has no multiple fibres (in the `infimum', not `gcd', sense) in codimension one.

\begin{remark}\label{Delta} If $f:X\to Y$ is a fibration, with smooth fibres $X_y$ such that $\nu_1(X_y)=-\infty$, it is not true in general that $\kappa_1^*(X)+dim(X)= \kappa_1^*(Y)+ dim(Y)$. This equality also fails for $\nu_1^*$. Indeed: Let $S$ be an Enriques surface with universal cover $T$ (a $K3$-surface), with $S=T/(t)$, where $t$ is the Enriques involution. Let $C$ be a hyperelliptic curve of genus $g>1$ with hyperelliptic involution $h$. Let $X':=C\times T$, we have: $\kappa_1^*(X')=1+(-2)=-1$ (see for example \cite{HP18} for the equality $\nu_1^*(S)=-2$ for a $K3$ surface. The weaker statement for $\kappa_1^*$ goes back to S. Kobayashi). On the other hand, let $(t\times h)$ be the involution acting without fixpoint on $X'$. Let $X:=X'/(t\times h)$. We thus have $\kappa_1^*(X)=\kappa_1^*(X')=-1$. But $X$ is equipped with a fibration $f:X\to C/(h)=\Bbb P_1:=B$ with smooth fibres $S$, so that $\kappa_1^*(X)+dim(X)=\kappa_1^*(X')+dim(X')=1+(-2)+2=1$, but $\kappa_1^*(B)+dim(B)=0$. The equality is however restored if we introduce the `orbifold base' $(\Bbb P_1,\Delta_f)$ of $f$, and define suitably $\kappa_1(B,\Delta)=\kappa(B, K_B+\Delta)$. In higher dimensions, one needs to use the sheaves $Sym^{[m]}(\Omega^1(B,\Delta))$ introduced in \cite{Ca04} to define $\kappa_1(X,\Delta)$.\end{remark}

\subsection{Divisors partially supported on the fibres of a fibration.}

\begin{definition}\label{PSF} Let $f:T\to W$ be a proper surjective holomorphic map with connected fibres (a `fibration')  between connected complex manifolds, and let $E\subset T$ be a reduced divisor.
We say that $E$ is `partially supported on the fibres of $f'$ if, for any divisorial irreducible component $D$ of $f(E)\subset W$, there exists an irreducible component of $f^{-1}(D)$ which is surjectively mapped onto $D$ by $f$, but is not contained in $E$.\end{definition}

\begin{example}\label{rcfibres} 1. In the situation of \ref{PSF}, let $D\subset W$ be an irreducible divisor, and let $f^*(D):=\sum_{k\in K} t_k.E_k+R$ be the scheme theoretic inverse image of $D$, where the $E_k's$ are distinct irreducible divisors surjectively mapped by $f$ onto $D$, and $R$ is $f$-exceptional. Let $t:=inf_{k}\{t_k\}$, and let $E_t:=\sum_kt.E_k$. Then $E:=f^*(D)-E_t-R$ is partially supported on the fibres of $f$ (and empty if $R=0$, and if $t_k=t,\forall k$).

2. One of the situations in which we will use this notion is the following one: assume that the fibres of $f$ are rationally connected and $T,W$ are projective. By \cite{GHS}, $f^*(D)$ contains at least a reduced component, for any $D\subset W$. Thus $f^*(D)-f^{-1}(D):=E_D$ is either partially supported on the fibres of $f$, or empty.

 \end{example} 

The divisors partially supported on the fibres of a fibration\footnote{The referee points out that this notion also appears in N. Nakayama's book \cite{Nak}, p.103 under the name of `divisor of insufficient fibre type', and that this book also contains the particular case of Lemma \ref{lemmapsf} when $\cG=\mathcal{O}_T$.} are used in the following way, inspired from \cite{Ca04} and slightly extending \cite{Ca16}, Lemmas 6.11,  Corollary 6.13 (where $\cF$ is of rank one, there).

\begin{lemma}\label{lemmapsf} Let $f:T\to W$ be a fibration, with $T,W$ complex projective smooth and connected. Let $\cF,\cG$ be  torsionfree coherent sheaves on $W,T$ respectively, and $E\subset T$ be partially supported on the fibres of $f$. There exists $s=s(\cG, E)>0$, such that, $\forall s'>s$ the natural inclusion map :
$H^0(T, f^*(\cF)(sE)\otimes \cG)\to H^0(T, f^*(\cF)(s'E)\otimes \cG)$ is surjective, hence bijective, where $f^*(\cF)(s'E):= f^*(\cF)\otimes \mathcal{O}_T(s'E)$. 
The constant $s=s(\cG,E)$ depends on $\cG$ and $E$, but not\footnote{The bound $s$ also depends on an auxiliary ample line bundle $H$ on $W$ if $dim(W)>1$.} on $\cF$.
\end{lemma}

\begin{proof} Let us now choose $A$ ample on $T$ such that $\cG\subset A^{\oplus R}$, for some $R$. 
It is thus sufficient to consider the case when $\cG=A$, which we do from now on. We may of course choose $A$ to be as ample as we wish, for example very ample.

Hartog's theorem shows that we may, and shall, assume that $E$ is of simple normal crossings, by suitably blowing $T$ up. Indeed: let $\pi:T'\to T$ be a modification such that the strict transform $E'$ of $E$ is of simple normal crossings, and partially supported on the fibres of $f':=f\circ \pi:T'\to W$. 
We have thus $\pi^*(E)=E'+\cE$, with $\cE$ $\pi$-exceptional.  We have: $\pi^*(H^0(T, \pi^*(f^*(\cF)\otimes s'E+A))=H^0(T', (f')^*(\cF)\otimes (s'E'+A'+s'\cE))=H^0(T', (f')^*(\cF)\otimes (s'E'+A'))$, by Hartog's theorem. This easily implies that Lemma \ref{lemmapsf} holds true for $(f,E)$ if it holds for $(f',E')$.

We first consider the special case when $T$ is a surface and $W$ a curve. We will reduce to this case after this.

Let thus $f:T\to W$ be a fibration from a surface to a curve. Thus $E$ is concentrated on finitely many fibres of $f$. By increasing $s$, we can easily reduce to the case where $E$ is connected and contained in a single fibre of $f$, case which we now consider. In this case, $E$ is exceptional in $T$, and $E=\sum_i E_i$, the $E_i$ being the components of $E$. By \cite{BPV}, Corollary I.2.11, we may equip the components $E_i$ of positive integral multiplicities $m_i>0$ such that, if $E^+:=\sum_im_i.E_i$, then we have: $E^+.E_i<0,\forall i$.

{\bf Claim 1:} There exists $s:=s(A,E)$ such that:

 $H^0(E^+,f^*(\cF)\otimes (s'E^++A))=0,\forall s'\geq s$.

 \begin{proof} (of Claim 1). Take $s>\frac{(A-m_i.E_i).E_i}{-E^+.E_i},\forall i$. We thus have: 
 
 $(s'E^++A-k.E_i).E_i<0,\forall s'\geq s, \forall k$, integer such that: $k\leq m_i$.
 
 We shall now show that: 
 
 $H^0(k.E_i,f^*(\cF)\otimes (s'E^++A))=0,\forall i, \forall k, 1\leq k\leq m_i$.
 
 Applying this with $k=m_i,\forall i$, we get the claim. Fix some $i$, and proceed inductively on $k$.
 
 For $k=1$ we need to show that $H^0(E_i,f^*(\cF)\otimes (s'E^++A))=0$. But this is clear since $f^*(\cF)\otimes (s'E^++A)$ is a sum of line bundles of negative degrees on $E_i$, because $f^*(\cF)$ is a sum of trivial line bundles on $E_i$.
 
The induction step from $1\leq k<m_i$ to $(k+1)$ is obtained similarly, tensoring with $f^*(\cF)\otimes (s'E^++A))$ the decomposition exact sequence (see \cite{BPV},II,\S 1, Equation (4), applied with $A=E_i, B=k.E_i$):

$0\to \mathcal{O}_{E_i}(-k.E_i)\to \mathcal{O}_{(k+1).E_i}\to \mathcal{O}_{kE_i}\to 0$, and applying $H^0$, we get the exact sequence:

$0\to H^0(E_i, f^*(\cF)\otimes (s'E^++A-k.E_i))\to H^0((k+1)E_i, f^*(\cF)\otimes (s'E^++A))\to H^0(k.E_i, f^*(\cF)\otimes (s'E^++A)),$ in which the two extreme terms, and so also the middle one, vanish, proving the claim.\end{proof}

\medskip

\begin{remark}1. The $m_i's$, and so $s$, depend only on the intersection matrix of the components of $E$ (and of course, on $A$ and $E$), that is: on the dual graph of $E$ and the self-intersections of the components of $E$ in $T$, since $E$ is of simple normal crossings.

2. The proof shows that if $E^{\bullet}:=\sum_ik_i.E_i$, where $k_i>0,\forall i$ are arbitrary, we still get $H^0(E^{\bullet},f^*(\cF)\otimes (s'E^++A))=0,\forall s'\geq s(\bullet)$, where  $s(\bullet)>\frac{(A-k_i.E_i).E_i}{-E^+.E_i},\forall i$ (this is not used in the sequel).
\end{remark}

We now reduce the proof of Lemma \ref{lemmapsf} for $f:T\to W$ with $dim(T)=n\geq 2,$ $ dim(W)=d\geq 1$ to the preceding claim. Take first the intersection of $(n-d-1)$ generic members of the linear system $A$ to get by restriction $f_t:T_t\to W$, which is a fibration in curves. If $d>1$, choose next a very ample line bundle $H$ on $W$, and consider a generic member $W_{s}$ of the complete intersection of $(d-1)$ generic members of the linear system $H$ . Then cut $f^{-1}(W_s)$ by $T_t$ to get by restriction $f_{t,s}:T_{t,s}\to W_s$, which is a fibration from a smooth surface to a smooth curve. The divisor $E_{t,s}:=E\cap T_{t,s}$ is partially supported on the fibres of $f_{t,s}$. By Bertini theorem, it is still of simple normal crossings. 

{\bf Claim 2:} There are multiplicities $m_i>0$ on the components $E_i$ of $E$ such that\footnote{We restrict here $E$ over an irreducible divisorial component $D$ of $f(E)$.}, for all generic $(t,s)$ as above, if we define $E^+:=\sum_im_i.E_i$, and $E^+_{t,s}:=E^+\cap T_{t,s}=\sum_im_i.E_{i,t,s}$, then $E^+_{t,s}.C'<0$, for each irreducible component $C'$ of $E_{t,s}$.

\medskip

We first prove that claim 2 implies the conclusion of Lemma \ref{lemmapsf}. 

We indeed have: $H^0(E^+,f^*(\cF)\otimes (s'E^++A))=0,\forall s'\geq s=s(A,E,H)$, since this is true by claim 1 and restriction to the $(E_{t,s})'s$ which cover a nonempty Zariski open subset of $E$.

From which follows by induction on $s'\geq s$ and the standard exact sequences for the $H^0$-cohomology, that:

 $H^0(T,f^*(\cF)\otimes (s'E^++A))\cong H^0(T,f^*(\cF)\otimes ((s'+1)E^++A)),\forall s'\geq s$.
 
 We thus have the statement of Lemma \ref{lemmapsf} for $E^+$. 
 
 Since $E\leq E^+\leq \mu.E$, , if $\mu:=max_i\{m_i\}$, this easily implies that $H^0(T,f^*(\cF)\otimes (s"E+A))\cong H^0(T,f^*(\cF)\otimes (s^+E+A)),\forall s"\geq s^+=s.\mu$, which is Lemma \ref{lemmapsf} for $E$. Indeed, we have, for any sheaf $\cH$, the following sequence of injective maps:
 
 $H^0(T,\cH\otimes sE^+)\to H^0(T,\cH\otimes (\mu.sE))\to H^0(T,\cH\otimes((\mu.s+j)E))\to H^0(T,\cH\otimes((\mu.s+j)).E^+)\to H^0(T,\cH\otimes(s+((m-1)s+j)E^+))$, and so, if its composition is bijective, each map in this sequence is bijective, for any $j\geq 0$. In particular, the second map in this sequence is bijective, if we choose $\cH=f^*(\cF)\otimes A$. This implies Lemma \ref{lemmapsf}.
 
 \medskip

 We still need to show the Claim 2. Let $D$ be a divisorial component of $f(E)$. We restrict $E$ over $D$. Let $E_{t,s}:=E\cap T_{t,s}$, and $E_{i,t,s}:=E_i\cap T_{t,s}$, so that $E_{t,s}:=\sum_iE_{i,t,s}$. Then $f_{t,s}:E_{t,s}\to W_s$ maps $E_{t,s}$ to a finite set $D\cap W_s$ of cardinality $D.H^{d-1}$. Let $w\in D\cap W_s$ be one of these points, and fix also some $(t,s)$. Write then $C_i:=E^w_{i,t,s}$ for the fibre of $E_{i,t,s}$ over $w$, and $C:=\sum_iC_i$. We can now further decompose each $C_i=\sum_{h}C_{i,h}$, were the $C_{i,h}$ are irreducible disjoint, since $E_{t,s}$ is of simple normal crossings in $T_{t,s}$, so that $C_i$ is smooth.
 
 We shall now prove the following two properties, which imply Claim 2, since the $m_i>0$ in property {\bf P2} depend only on the $C_i^2$ and on the dual graph of $C$, which are locally constant in $(t,s,w)$, hence independent of $(t,s, w)$:
 
 {\bf P1:} If $i\neq k\in I$, then $C_k.C_{i,h}=C_k.C_{i,h'}, \forall h,h'$, and $C_{i,h}^2$ do not depend on $h, t,s,w$, for each given $i$.
 
 {\bf P2:} Let  $C^+:=\sum_im_i.C_i$, where the $m_i's$ are positive integers such that $C^+.C_i<0,\forall i\in I$. Then $C^+.C_{i,h}<0,\forall i,h$, and these $m_i$ satisfy these inequality for each $(t,s,w)$.

 Proof of {\bf P1:} Let $f:E_i\to D$ have Stein factorisation $g:E_i\to D', u:D'\to D$, with $d$ finite of geometric degree $\delta$ and $g$ connected (i.e: with connected fibres), and $f=u\circ g$. For any $w\in D$ generic, let $D'_w:=\sum_hD_{w,h}$ the finite fibre of $d$ over $w$. Cutting $T$ by $(n-d-1)$ generic members of $A$ to get $T_t$, we restrict $f$ to $T_t$, and get $f_t:E_i\cap T_t\to D$, which has Stein factorisation $f_t=u\circ g_t$, where $g_t:E_i\cap T_t\to D'$ is the restriction of $g$ to $E_i\cap T_t$. The fibres of $g_t:E_i\cap T_t\to D'$ are, by Bertini theorem, irreducible curves $E_{i,t,h,w}$ which build an algebraic family parametrised by the irreducible variety $D'$. The curves $E_{i,t,h,w}$ are thus  cohomologically equal, and we have, for each of these curves $E_{i,t,h,w}$, for $w\in D$, $(w,h)\in D'$ the equality: $\delta.[E_{i,t,h,w}]=[E_i].[A]^{n-d-1}.[F])$, where $F$ is any fibre of $f$, since $\sum_hE_{i,t,h,w}=E_i\cap f_t^{-1}(w)\cap T_t$.
 
 Restricting $T_t$ over $W_s$, intersection of $(d-1)$ generic members of the linear system $H$ on $W$, preserves these properties, except then that $w\in D_s:=D\cap W_s$.

 Let $N:=[D].[H]^{d-1}$, so that $N.[F]=f^*([D].[H]^{d-1})$ is the number of fibres of $f_{t,s}:T_{t,s}\to W_s$ containing components of $E_{t,s}$. For any $(t,s, w\in D_s)$, the intersection numbers\footnote{Recall that $C_{i,h}=E_{i,t,s,h,w}$, for any $(t,s,h,w)$.} $N.\delta.C_{i,h}^2$ are thus equal to the self-intersection of the reducible curve $E_{i,t,s,w}$ in $T_{t,s}$, that is: $[E_i\cap T_{t,s}]^2=[E_i]^2.A^{n-d-1}.f^*([H]^{d-1})$, since $C_{i,h}.C_{i,h'}=0,\forall h\neq h'$, the curves $C_{i,h}$ and $C_{i,h'}$ being disjoint. This shows the second property claimed by {\bf P1}

 In the same way, if $k\neq i\in I$: $[E_k].[E_i].[A]^{n-d-1}.f^*([H]^{d-1})=N.\delta.C_k.C_{i,h},\forall h$, which shows the first property claimed by {\bf P1}.
 
 The independence on $(t,s,w)$ is clear by the cohomological description of all these numbers.

 Proof of {\bf P2:} Recall that $\delta$ is the geometric degree of the preceding map $d:D'\to D$, that is: the number of the (disjoint) irreducible components $C_{i,h}$ of $C_i$. 
 
 By assumption: $\forall i:0>(\sum_km_k.C_k).C_i=\sum_{k\neq i}m_kC_k.C_i+m_i.C_i^2$
 
 But: $(\sum_{k\neq i}m_kC_k).C_i=\delta.\sum_{k\neq i}m_k.C_k.C_{i,h}, \forall h$, by {\bf P1}, and:
 
$m_i.C_i^2=\delta.m_i.C_{i,h}^2, \forall h$, since $C_{i,h}.C_{i,h'}=0$ if $h\neq h'$, and $C_{i,h}^2$ does not depend on $h$.

Thus: $0>(\sum_{k}m_kC_k).C_i=\delta.(\sum_{k}m_kC_k).C_{i,h},\forall h$, which is {\bf P2}.
\end{proof}

\section{Submanifolds of Abelian varieties and Projective spaces.}\label{Ab}

 When $X$ is of general type, no algebro-geometric pattern is known to describe $\kappa_1(X)$, even for surfaces. In the following two cases, Sakai gave a simple (but `extrinsic', depending on an embedding) description.

\begin{proposition}\label{pab}(\cite{Sak}, Theorem 7) Let $A$ be an Abelian variety and $X\subset A$ be a submanifold with ample normal bundle. Then: $\kappa_1(X)=min\{n,c\}$, where $n$ is the dimension of $X$, and $N$ is the dimension of $A$, and $c:=N-n$ the codimension of $X$ in $A$.\end{proposition}

From Remark \ref{rnu=k"}, we thus deduce that Conjecture \ref{ConjO} is true when $min\{n,c\}\geq (n-1)$, that is, when $c\geq (n-1)$. Thus Conjecture \ref{ConjO} is true for surfaces with ample normal bundles in Abelian varieties. 

\begin{theorem}\label{tab}If $X$ is as in Proposition \ref{pab}, we have: $\nu_1(X)=\kappa_1(X)$.
\end{theorem}

\begin{proof} If $c\geq n$, whe have $\kappa_1(X)=n$, hence $\nu_1(X)=n\leq c$.

If $n<c$, the next more general Lemma shows that $\nu_1(X)\leq c$. We thus have $\nu_1(X)=\kappa_1(X)=c$.

\begin{lemma} Let $X\subset A$ be a submanifold of dimension $n$ in an Abelian variety of dimension $N$. Then $\nu_1(X)\leq N-n:=c$.
\end{lemma}

\begin{proof} Let $H$ be an ample line bundle on $A$, and $N$ the normal bundle of $X$ in $A$. We tensorise by $H$ the Koszul-type sequence constructed in \cite{Sak}, proof of Theorem 7. We thus get a spectral sequence with $E_1$-terms $E_1^{-q,q}=H^q(X, \wedge^qN^*\otimes Sym^{m-q}(\Omega^1_A)\otimes H), 0\leq q\leq m,$ abutting to $H^0(X,Sym^m(\Omega^1_X)\otimes H)$. 

Since $h^0(X,\Sym^m(\Omega^1_X)\otimes H)=\sum_{q=0}^{q=m} E_{\infty}^{-q,q}$, for any $m\geq 0$, and $dim(E_1^{-q,q})\geq dim(E_{\infty}^{-q,q}),\forall m, q$, we have (because $\wedge^qN^*=0$ if $q>c$): 

$h^0(X,Sym^m(\Omega^1_X)\otimes H)\leq \sum_{q=0}^{q=c}h^q(X,\wedge^qN^*\otimes Sym^{m-q}(\Omega^1_A)\otimes H)$

$=\sum_{q=0}^{q=c}H^q(X,\wedge^qN^*\otimes H).rank(Sym^{m-q}(\Omega^1_A))$ 

$\leq (\sum_{q=0}^{q=c}H^q(X,\wedge^qN^*\otimes H)).rank(Sym^m(\Omega^1_A))\leq C. m^{N-1},$

for a suitable constant $C>0$. Hence $\nu_1(X)\leq N-1-(n-1)=c$.
\end{proof}\end{proof}

\begin{proposition}\label{N}(\cite{Sak}, Theorem 8) Let $X\subset \Bbb P_N$ be a smooth complete intersection. Assume that $2n>N$, then: $\kappa_1(X)=-\infty$, $n$ being the dimension of $X$.
\end{proposition}

\begin{theorem}\label{tN} If $X$ is as in Proposition \ref{N}, then $\nu_1(X)=-\infty$.
\end{theorem}

\begin{proof} Fix $k\geq 0$, let $H:=\mathcal{O}_{\Bbb P^N}(1)$,and write $S^m:=Sym^m(\Omega^1_{\Bbb P^N\vert X})$. 

As before, we have: $$h^0(X,Sym^m(\Omega^1_X)\otimes H^k)\leq \sum_{q=0}^{q=m}H^q(X,\wedge^qN^*\otimes S^{m-q}.$$
By the same reference \cite{KO}, p.521 as in \cite{Sak}, we next have: $$H^i(X,S^m\otimes H^t)=0,\forall i<n, m\geq t+2.$$

Since $\wedge^qN^*$ is a direct sum of $q$ line bundles of the form $H^s$ for $s\leq -q$, and is zero if $q>c:=N-n$, we get, for each $q\leq c$ and any $k\geq 0$: $E_1^{-q,q}=0$ if $m\geq k+2$, and so $h^0(X,Sym^m(\Omega^1_X)\otimes H^k)=0$ for $m\geq k+2$. This means that $\nu_1(X)=-\infty$, as claimed.
\end{proof}

\begin{remark}\label{rem ample} 1. When $2n<N$ in Proposition \ref{N}, and when the degree of $X$ is sufficiently large, $\Omega^1_X$ is ample for the generic member $X$, by \cite{BD}. Conjecture \ref{ConjO} is thus true in this case by Remark \ref{rnu=k"}, and the codimension inequality is then optimal. 

2. If $X$ is a surface, and $(c_1^2-c_2)(X)>0$, then $\kappa(X, \Omega^1_X)=2$ (Bogomolov), and Conjecture \ref{ConjO} is thus satisfied by Remark \ref{rnu=k"}.

However, already in the limit case $(c_1^2-c_2)(X)=0$, nothing seems to be known in general, not even for complete intersection surfaces in $\Bbb P_N$ (see the list of possible multidegrees in \cite{Sak}). 

When $X$ is an ample divisor in an Abelian threefold $A$, we have $(c_1^2-c_2)(X)=0$ and $\kappa_1(X)=1$, and Conjecture \ref{ConjO} is thus true, by either the preceeding Proposition \ref{pab} and Remark \ref{rnu=k"}, or by Theorem \ref{tab}.

In this case however, $\pi_1(X)=\pi_1(A)$ is infinite. It would be interesting to know the possible values of $(\kappa_1(X), \pi_1(X))$ for surfaces with $(c_1^2-c_2)=0$.

3. When $\kappa_1(X)=-\infty$, any $X$, it is shown in \cite{BKT} that every linear representation of $\pi_1(X)$ in any $Gl(N,K)$, for any field $K$, has finite image. One may thus wonder whether $\pi_1(X)$ itself should be finite, then. This is true at least in the following cases: 

3.a. $X$ is smooth of dimension $n$ in $\Bbb P_N$ with $2n>N$, by \cite{BL}.

3.b. $X$ is a surface with $3c_1^2<c_2$, by the classification in  \cite{R}.

3.c. $\kappa_p(X)=-\infty,\forall p>0$, by \cite{BC}. 

\end{remark}

\section{Klt varieties with $c_1=0$ and their smooth models.}\label{Sc_1=0}

\subsection{Comparing $\nu_1$ and $q^+$.}

The following proposition has also been established in \cite{Gac}, Theorem 1.2, to which we refer for more results and interesting examples (see also Remark \ref{r}(6) below).

\begin{proposition}\label{carab} Let $X'$ be a normal connected complex projective variety of dimension $n$ with klt singularities and trivial numerical class, and let $X$ be a smooth model of $X'$. 

Then: $\kappa^*_1(X)\leq \nu^*_1(X)\leq q^+(X')-n$, with equality after replacing $X'$ by a suitable finite {\bf quasi-\'etale} cover, where $q^+(X')\leq dim(X)$ is the maximal `quasi-\'etale' irregularity of $X'$. In particular: $\nu_1(X)=-\infty$ if $q^{+}(X')=0$.
\end{proposition}

\begin{proof} By the singular Bogomolov-Beauville decomposition (see \cite{HP18} and  the references there, or \cite{Ca20} in the projective case, and \cite{BGL} in the K\"ahler case), $X'$ admits a finite quasi-\'etale cover $X"=Z'\times A$ which is a product, with $Z'$ a product of Hyperk\"ahler and Calabi-Yau varieties,  and $A$ an Abelian variety of dimension $q^+(X')=q^+(X")$. By \cite{HP18}, $\nu_1(Z)=-\infty$ if $Z'$ is either Hyperk\"ahler or Calabi-Yau, or a product of such, and $Z$ a smooth model of $Z'$. Thus, by the behaviour of $\kappa^*_1,\nu^*_1$ under products, $\kappa^*_1(X")=\nu^*_1(X")=\kappa_1(A)+(q^+(X')-n)=q^+(X")-n$. Since $\nu_1^*(X')\leq \nu_1^*(X")$, we get the claimed inequality.
\end{proof}

\subsection{Comparing $\nu_1$ and $q'$.}

Without passing to a suitable finite quasi-\'etale cover, we have a similar statement, replacing $q^+$ by $q'$, but the proof, although elementary, becomes much longer, because of singular quotients $A/G$ of compact complex tori by finite group actions, treated in the next section.

\begin{definition} If $X$ is a normal connected projective complex space with rational singularities, we define $q'(X)$ as the maximum (possibly $+\infty$) of the Albanese variety of $X$, or equivalently of the rank of the abelianisation of the finite index subgroups of $X$. \end{definition}

We thus have $q'(X)\leq q^+(X)$, with strict inequality possible, for example when $X$ is a singular Kummer surface, in which case $q^+=2$, and $q'=0$. When $X$ is smooth, $q^+(X)=q'(X)$, since quasi-\' etale covers are \' etale. For curves of general type, we have $q'(X)=+\infty$. Both $q^+$ and $q'$ are bounded by $n=dim(X)$ if $\kappa(X)=0$, or more generally if $X$ is special. 

We gather some properties of $q'$ used in the sequel in the following Proposition \ref{pq'}.

\begin{proposition}\label {pq'} Let $X'$ be a normal connected projective complex space with klt singularities. Then:

1. $X$ has rational singularities.

2. $\pi_1(X)=\pi_1(X')$ if $X$ is any smooth model of $X'$. In particular: $q'(X)=q'(X')$. 

3. $q^+(Z')=q'(Z')=0$ if $Z'$ is a product of Hyperk\"ahler and Calabi-Yau varieties\footnote{It is conjectured that $\pi_1(X')$ is finite.}.

\end{proposition}

\begin{proof} Claim 1 (resp. Claim 2) follows from \cite{KM}, Theorem 5.22 (resp. \cite{Tak}, Theorem 1.1). Claim 3 follows from the singular Beauville-Bogomolov decomposition \cite{HP18}.\end{proof}

\begin{lemma}\label{lq'} Let $f:X\to A$ be surjective map from $X$, normal projective with klt singularities, on the Abelian variety $A$, with generic fibres connected such that $q=0$. Let $G$ be a finite group acting holomorphically and effectively on $X$. Let $\pi:X\to X/G$ be the natural projection. Then $f=a_X$ is the Albanese map of $X$, and there is so an induced action of $G$ on $A$, and  quotient maps $g:X/G\to A/G$ and $p:A\to A/G$  such that $g\circ \pi=p\circ f$. 

1. Then $Alb(X/G)=Alb(A/G)$. In particular, $q(X/G)=q(A/G)$.

2. If $q'(X/G)=q(X/G)$, then $q'(X/G)=q'(A/G)$.

\end{lemma}

\begin{proof} 1. There exists a natural map $a_g:Alb(X/G)\to Alb(A/G)$ such that $a_g\circ a_{X/G}=a_{A/G}\circ g$. Since the generic fibres of $g$ are images of fibres of $f$, they have $q=0$, and so are mapped to points by $a_{X/G}$, and the map $a_{X/G}$ factorises through $A/G$.The universal property of the Albanese map gives the conclusion.

2. Since we have a dominant map $g:(X/G)\to (A/G)$, we have: $q'(X/G)\geq q'(A/G)$. The assumptions and Claim 1 thus imply that: $q(A/G)=q(X/G)=q'(X/G)\geq q'(A/G)\geq q(A/G)$. All $4$ terms are thus equal.\end{proof}

\begin{theorem}\label{c_1=0}Let $X'$ be a normal connected complex projective variety of dimension $n$ with klt singularities and numerically trivial canonical class, and let $X$ be a smooth model of $X'$. Then:

1.  $\kappa^*_1(X)=\nu^*_1(X)=q'(X)-n$. 

In particular: 

2. $\kappa_1(X)=\nu_1(X)=0$ if and only if $X'$ is smooth and is an \'etale quotient of an Abelian variety.

3. $\kappa_1(X)=\nu_1(X)=-\infty$ if and only if $q'(X)=0$.

4. Assume $X$ is smooth compact K\"ahler with $\kappa(X)=0$. Assume the existence of a birational model of $X$ (with klt singularities and $c_1=0$). Then Claims 1,2,3 hold true for $X$. \end{theorem} 

\begin{remark} In the situation of Claim 4, Claim 2 strengthens \cite{Mis}, in which the assumption $\kappa_1(X)=0$ is replaced by the assumption that $Sym^m(\Omega^1_X)$ is generated by its sections for large $m$. 
\end{remark}

We shall show below how Theorem \ref{c_1=0} follows from the Beauville-Bogomolov decomposition, and the special case of quotients of Abelian varieties by finite groups $A/G$ (i.e: Theorem \ref{thq"}, proved in the next section):

\begin{theorem}\label{thq"} For $A/G$ a quotient of an Abelian variety $A$ by a finite group $G$, we have: 
$\nu_1^*(A/G)=\kappa_1^*(A/G)=q'(A/G)-n$.
\end{theorem}

\begin{proof} (of Theorem \ref{c_1=0}, assuming Theorem \ref{thq"}) Claims 2 and 3 are immediate consequences of Claim 1: $\kappa_1^*=\nu_1^*=q'-n$, which we now prove. 

We may assume that $q'(X')=q(X')$ by replacing $X'$ with a suitable finite \'etale cover since $\nu_1^*,\kappa_1^*, q'$ and $dim(X)$ are invariant under finite \'etale covers. After another finite and quasi-\'etale cover $X_1$ with canonical singularities of $X'$, and Galois of group $G$, we have a product $X'_1=Z'\times A$, with $q^+(Z')=q'(Z')=0$, and so $X'=X_1/G$. Since we have the projection $f:X_1\to A$ with fibre $Z'$, we can apply Lemma \ref{lq'} since the Albanese map is surjective with connected fibres for this class of varieties, and then conclude that $q'(X')=q'(A/G)$. 

Now $Aut(Z')$ is discrete (since $Z'$ is not uniruled). We thus have: $Aut(X'_1)=Aut(Z')\times Aut(A)$, by \cite{beauv}, Lemma, p.8 (the proof of which works in the singular case as well). We thus have an injection: $G<G'\times G_1$, where $G'<Aut(Z'),G_1<Aut(A)$ are finite groups of automorphisms of $Z'$ and $A$ respectively, images of the projections of $G$. From this injection, we get a natural finite quotient map:  $h:X'\to (Z'/G')\times (A/G_1)$. We may, and shall, assume that $G$ is isomorphic to $G_1$ by first dividing with the kernel $K$ of the projection $G\to G'$, and replacing $Z'$ by $Z'/K$ in the sequel. 

From the quotient map $h:X'\to (Z'/G')\times (A/G)$, we  deduce that $\nu_1^*(X')\geq \nu_1^*(Z'/G')+\nu_1^*(A/G)$. Since $\nu_1(Z')=\nu_1(Z'/G')=-\infty$ by Proposition \ref{carab}, we deduce that $\nu_1^*(X')\geq \nu_1^*(A/G)+\nu_1^*(Z')=\nu_1^*(A/G)-dim(Z')$. Theorem \ref{thq"}, proved in the next section, asserts that $\nu_1^*(A/G)=q'(A/G)-dim(A)$. Thus $\nu_1^*(X')\geq q'(A/G)-dim(A)-dim(Z')=q'(A/G)-dim(X')=q'(X')-dim(X')$, since $q'(X')=q'(A/G)$. Since Remark \ref{rem'q'}(3) below shows that $\nu_1^*(X')\leq q'(X')-dim(X')$ is always true for this class of varieties, we get that $\nu_1^*(X')=q'(X')-dim(X')$, as claimed.
The same argument shows that $\kappa_1^*(X')=q'(X')-n$, concluding the proof.\end{proof}

\subsection{Comparing $\nu_1$ and $\kappa$.}

Without assuming the existence of good minimal models, we can prove only a much weaker statement than \ref{c_1=0}.

\begin{proposition}\label{cab'} Assume $X$ is smooth and complex projective with $\nu(X)=0$. Then $\kappa_1(X)\leq \kappa(X)=0$.\end{proposition}

\begin{proof} $K_X$ is pseudoeffective, since $\nu(X)=0$. Thus, by \cite{CP19}, Theorem 7.3, for any rank-one subsheaf $L\subset (\otimes^m \Omega^1_X)$, we have: $\nu(X,L)\leq \nu(X)=0$. 
For any $m>0$, if $E_m\subset Sym^m(\Omega^1_X)$ is the subsheaf generated by the global sections of $Sym^m(\Omega^1_X)$, the rank $r_m$ of $E_m$ is equal to $h^0(X, \Sym^m(\Omega^1_X))$, by the previous inequality $\nu(X,L)\leq 0$, applied to $L:=det(E_m)\subset (\otimes^N \Omega^1_X)$ for some suitable $N$, since the inequality $r_m<h^0(X,Sym^m(\Omega^1_X))$ would produce at least $2$ linearly independent sections of $L$. So $r_m$ grows as $m^{n-1+\kappa_1(X)}$ as $m$ goes to $+\infty$. On the other hand, $r_m\leq rk(Sym^m(\Omega^1_X))$, which grows like $m^{n-1}$. Thus $\kappa_1(X)\leq (n-1)-(n-1)=0$. The last equality $\kappa(X)=0$ if $\nu(X)=0$ is proved in \cite{Kaw}.\end{proof}

More generally, Conjecture $\Lambda_0$ in \cite{Sak} claims that: $\kappa_1(X)\leq \kappa(X)$ if $\kappa(X)\geq 0$. Conjecture \ref{ConjO} implies the following strengthening.

\begin{conjecture}\label{cjnu} 1. If $K_X$ is pseudoeffective, $\nu_1(X)\leq\kappa(X)$. 

2. If $X$ is uniruled, and $r:X\to R_X$ is its MRC fibration, then: $\nu_1(X)\leq \kappa(R_X)$.
\end{conjecture}

\begin{remark} When $X$ is projective, Claim 1 and Theorem \ref{tRQ} imply Claim 2 .
\end{remark}

\begin{proposition}\label{birstab}Assume $\kappa(X)\geq 0$, and $\nu(X_z)=0$ for the generic fibre\footnote{Since we are dealing with bimeromorphic invariants, we assume that $\varphi$ is holomorphic.} $X_z$ of the Moishezon-Iitaka fibration $\varphi:X\to Z$ of $X$. 

Then: $\kappa_1(X)\leq \kappa(X)$.
\end{proposition}

\begin{proof} We have (see \cite{Sak}, Theorem 4): $\kappa_1(X)\leq dim(Z)+\kappa_1(X_z)=\kappa(X)+\kappa_1(X_z)\leq \kappa(X)$, since Proposition \ref{cab'} and $\nu(X_z)=0$ imply that $\kappa_1(X_z)\leq 0$.
\end{proof}

\section{Torus quotients.} \label{Storusq}

The main objective of this section is to prove Theorem \ref{thq"}, and then the structure Proposition \ref{propdeco}. The proof of Theorem \ref{thq"} is long, but elementary. More conceptual arguments might permit to shorten it considerably.

\subsection{An elementary lemma on group actions.}\label{groupaction}

 Let $G$ be a multiplicative group acting on a set $X$. For $g\in G$, the fixpoint set of $g$ is $X^g:=\{x\in X\vert g.x=x\}$. 
 
  We say that $G$ acts on $X$ without fixpoints if $X^g=\emptyset,\forall g\neq 1_G$.
 
 \medskip
 
  We shall split the action of $G$ into two parts: the action induced by the normal subgroup $F'$ generated by elements having fixpoints, and the quotient action of the quotient group $G/F'$ on the set of $F'$-orbits of $X$, which is shown to have no fixpoints. This splitting is essential in our study of quotients $A/G$ of Abelian varieties by finite groups of automorphisms.
 
\medskip

Let $E:=\{g\in G\vert X^g=\emptyset\}\subset G$, and let  $F:=G\setminus E\subset G$ be its complement. These sets are invariant under conjugation in $G$. 

Let $E',F'$ be the subgroups of $G$ generated by $E,F$ respectively. These are thus normal subgroups of $G=E'\cup F'$. 

\medskip

We first have the classical lemma (proof omitted):

\begin{lemma} Either $G=E'$, or $G=F'$.
\end{lemma}

Next:

\begin{lemma}\label{lnofix} Assume that $G=E'$. Let $Y:=F'\setminus X$ be the set of orbits of $X$ under the action of $F'$. Then $G/F'$ acts on $F'\setminus X$ without fixpoints. 
\end{lemma}

\begin{proof} For $g\in G$, let $\bar{g}:=g.F'\in G/F'$ be its image under the quotient $G\to G/F'$. In the same way, write $\bar{x}:=F'.x$ for the image of $x$ under the quotient $X\to F'\setminus X$.
Let $x\in X$, and denote by $S_x<F'<G$ the stabiliser of $x\in G$, and $S_{\bar{x}}<G/F'$ the stabiliser of $\bar{x}$ in $G/F'$. We have: $\bar{g}\in S_{\bar{x}}$ if and only if $g.x=f.x$ for some $f\in F'$, or equivalently, if $g^{-1}f.x=x$, that is, if $g^{-1}.f\in S_x<F'$. Thus $g^{-1}=((g^{-1}f).f^{-1}\in F'$, and $\bar{g}=1_{G/F'}$, as claimed.
\end{proof}

\begin{remark} I did not find a reference for this certainly standard lemma, which plays a crucial role in the next considerations.
\end{remark}

Let us also notice the following easy fact:

\begin{lemma}\label{F^*vsF} With the preceeding notations $G,F,F'$, if $F^*\subset F$ consists of the elements of primary order (i.e: power of a prime number), then $F'$ is generated by $F^*$.
\end{lemma}
\begin{proof} A cyclic group generated by $g$ is the direct product of its primary components, which all have at least as much fixpoints as $g$.
\end{proof}

\subsection{Smooth models of torus quotients.}

Let $A$ be an $n$ dimensional compact complex torus (i.e: a quotient $\C^n/\Lambda$ of $\C^n$ by a cocompact lattice $\Lambda$). Let $G$ be a finite group of complex affine automorphisms of $A$, $\pi:A\to A/G$ the quotient map, and $\rho:X\to A/G$ any desingularisation of $A/G$. By \cite{K}, $\rho_*:\pi_1(X)\to \pi_1(A/G)$ is an isomorphism (by a local version of Serre's covering trick, and rationality of quotient singularities). The Albanese map of $A/G$ is thus well-defined, and coincides with the one of $X$ up to composition with $\rho$. In particular, the finite \'etale covers of $X$ and $A/G$ naturally correspond. Moreover, for any finite \'etale cover $(A/G)'$ of $A/G$, the Albanese map $a':(A/G)'\to Alb((A/G)')$ of $(A/G)'$ is surjective with connected fibres (this is easy to see directly, and follows more generally from the fact that $(A/G)'$ is `special'). Thus: $q((A/G)')\leq n$, for each such $(A/G)'$, $q=dim(Alb((A/G)')$ being the irregularity.

Recall the:

\begin{definition} \label{defq'} Let $q'(A/G)$ be the maximum of all $q((A/G)')$ as $(A/G)'$ runs over all finite \'etale covers of $A/G$. 
\end{definition}

\begin{remark} \label{remq'} Recall that the invariant $q'$ should be carefully distinguished from the invariant $q^+(A/G)$ (see footnote 8). Easy examples show that the invariant $q'(X)$ takes on all possible values $0,1,...,n$ for $X$ such that $\kappa(X)=0$, hence $X$ special.\end{remark}

Recall that the invariants $\kappa_1^*$ and $\nu_1^*$ are still defined for $X'$ singular irreducible, by taking their values on any smooth model $X$ of $X'$, in particular when $X'=A/G$.

\begin{remark} \label{rem'q'} The following properties of $q',\kappa_1^*$ and $\nu_1^*$ are easily checked:

1. Preserved by bimeromorphic equivalence and finite \'etale covers.

2. $\nu_1^*(A/G)\leq \nu_1^*(A)=\kappa_1^*(A)=0$, since $\Omega^1_A$ is trivial. 

3. Let $a:A/G\to Alb(A/G):=B$ be the Albanese map. Then $\kappa_1^*(A/G)\geq 0-(n-dim(B))=dim(B)-n=q(A/G)-n$, since the sections of $Sym^m(\Omega^1_B)$ lift to sections of $Sym^m(\Omega^1_X)$, $X$ a smooth model of $A/G$, for any $m>0$. 

4. $\kappa_1^*(A/G)\geq q'(A/G)-n$: choose $(A/G)'\to (A/G)$ \'etale such that $q((A/G)')=q'(A/G)$, then apply property 3. \end{remark} 

The purpose of this section is to establish the following strengthened reverse inequality, which says in particular that $\nu_1$ and $\kappa_1$ are topological invariants for this class of varieties. These properties extend to $klt$ projective (or compact K\"ahler) varieties with $c_1=0$, as seen in the previous section, using the singular Bogomolov-Beauville-Yau decomposition. Let us recall the statement (of Theorem \ref{thq"}).

\begin{theorem}\label{thq'} For $A/G$ as above, we have: 

$\nu_1^*(A/G)=\kappa_1^*(A/G)=q'(A/G)-n$.
\end{theorem}

In the extreme cases $q'=0,n$, we have various characterisations. 

\begin{corollary}\label{q'=n} For $A/G$ as above, the following are equivalent:

1. $\kappa_1^*(A/G)=0$

2. $\nu_1^*(A/G)=0$

3. $q'(A/G)=n$

4. $\pi:A\to A/G$ is \'etale.

5. $A/G$ is smooth and $\kappa(A/G)=0$.

In particular: $\kappa^*_1(A/G)=0$ if and only if $A/G$ is a finite \'etale quotient of a compact complex torus.
\end{corollary}

\begin{proof} All equivalences between 1-2-3-4 are immediate consequences of Remark \ref{rem'q'}, except for $1\Longrightarrow 4$ which requires only a much weaker form of Theorem \ref{thq'}, namely Lemma \ref{sing}. The implication $4\Longrightarrow 5$ is obvious, the reverse implication $5\Longrightarrow 4$ again follows from Lemma \ref{sing} since if $\pi:A\to A/G$ is not \'etale, $G$ contains an element $g\neq 1$ with fixpoint set nonempty and of codimension at least $2$ since $\kappa(A/G)=0$, contradicting Corollary \ref{sing}. 
\end{proof}

\begin{corollary}\label{sing} $G$ contains $g\neq 1$ such that $g(a)=a$ for some $a\in A$, if and only if $\kappa_1^*(A/G)<0$. 
\end{corollary} 

Corollary \ref{sing} follows immediately from Theorem \ref{Cyclic}, essential step in the proof of Theorem \ref{thq'}.

\begin{corollary}\label{q'=0} For $A/G$ as above, the following are equivalent:

1. $\kappa_1^*(A/G)=-n$

2. $\nu_1^*(A/G)=-n$

3. $q'(A/G)=0$

4. $\pi_1(A/G)$ is finite.

\end{corollary}

\begin{proof} The equivalences between 1-2-3 are immediate from Theorem \ref{thq'}. The implication $4\Longrightarrow 3$ is obvious, the reverse implication follows from the next Lemma \ref{aab}, which applies with $s=n$.\end{proof}

\begin{lemma}\label{aab} Let $\pi_*:\pi_1(A)\to \pi_1(A/G)$ be the natural map. Then $2q'(A/G)=2(n-s)$ is the rank over $\Bbb Z$ of its image, which
 is Abelian, normal in $\pi_1(A/G)$, and of finite index dividing the order of $G$.\end{lemma} 

\begin{proof} The last two claims are a direct consequence of \cite{Ca91}, Proposition 1.3. We replace $A/G$ by its Galois \'etale cover $(A/G)'=A/G'$ such that $\pi_1(A/G')$ is $\pi_*(\pi_1(A))$. The existence of $G'\subset G$ achieving this equality follows from the fact that the map $\pi:A\to A/G$ lifts to $\pi':A\to (A/G)'$. We thus assume that $\pi_*$ is surjective, and that $\pi_1(A/G)=\pi_1(A)/K$, where $K$ is a free abelian subgroup of $\pi_1(A)$ of $\Bbb Z$-rank $r$. Thus $\pi_1(A/G)$ is the direct sum of a torsion group and of a free abelian group of even $\Bbb Z$-rank $2n-r$, $r=2s$. Any finite index subgroup of $\pi_1(A/G)$ has thus the same structure, with the same rank $r=2s$, and so $q'(A/G)=n-s$. \end{proof}

\subsection{Proof of Theorem \ref{thq'}.} We now start the proof of Theorem \ref{thq'}. We shall assume that $G$ does not contain non-trivial translations. This is no restriction, and does not alter neither the hypothesis, nor the conclusions below.

We make a first reduction.

\begin{lemma}\label{F^*} We may (and shall) assume that $G$ is generated as a group by $F^*$, the set of its elements which have fixpoints and primary orders.
\end{lemma}

\begin{proof}We use the notations of \S\ref{groupaction}. Let $E\subset G$ be the set of $g\in G$ acting without fixpoint on $A$, and  let $F$ be its complement in $G$. Let $E',F'$ be the (normal) subgroups of $G$ generated by $E,F$ respectively. Recall that either $G=E'$, or $G=F'$. In the first case, $G/F'=E'/F'$ acts without fixpoints on $A/F'$, in other words, the natural quotient map $A/F'\to A/G$ is \'etale, and the invariants $\kappa_1,\nu_1,q'$ coincide for $A/F'$ and $A/G$. Moreover, (see Lemma \ref{F^*vsF}), $F'$ is generated by $F^*$, the set of its elements of primary orders (which have fixpoints, too).\end{proof}

Theorem \ref{thq'} will be proved by combining the following two lemmas \ref{primary}, \ref{G-neral}, proved in the next two subsections.

Let $A^0=Aut^0(A)$ be the group of translations of $A$. Any choice of $a\in A$ defines an isomorphism between the group $A^0$ and $A$, with any given $a$ as zero element of the addition on $A$.

The first Lemma \ref{primary} proves Theorem \ref{thq'} when $G$ is cyclic of primary order, with fixed points.

\begin{lemma}\label{primary} Let $G$ be cyclic, generated by $g$, which has fixpoints on $A$. Let $C_g:=Im(g-1_A)<A^0$: this is a connected Lie subgroup. Let $a_g:A\to B_g:=A/C_g$ be the quotient map.

1. $C_g$ is preserved by the action of $g$, and $g$ acts trivially on $B_g$. 

2. $a_g:\widetilde{B_g}\to B_g$ is an isogeny, if $\widetilde{B_g}\subset A$ is any component of the fixpoint set of $g$. 

3. The addition map: $C_g\times \widetilde{B_g}\to A$ is a $g$-equivariant isogeny.

4. It induces a finite \'etale morphism $q_g: (C_g/G)\times \widetilde{B_g}\to A/G$. 

5. If $g$ is of primary order, then $\nu_1(C_g/G)=-\infty$ and $q'(C_g/G)=0$. 

6. $\nu_1^*(A/G)=\kappa_1^*(A/G)=q'(A/G)-n=-dim(C_g)$.

7.  $\pi_1(C_g/G)$ is finite, of exponent at most the order of $g$, if primary. 

\end{lemma}

The first 4 claims are elementary, essentially consequences of Bezout Theorem (see Lemma \ref{g-1}), the main claim is Claim 5, proved in Theorem \ref{Cyclic}. Claim 7 is established in Lemma \ref{tpi_1}.

\begin{corollary} If $A$ is simple (i.e: has no non-trivial subtorus), then, either $G$ acts fixpoints freely and $A/G$ is an \'etale quotient of $A$, or $\nu_1^*(A/G)=-n, q'(A/G)=0$, and $\pi_1(A/G)$ is finite.
\end{corollary}

Lemma \ref{G-neral} then permits to deal with the general case.

\begin{lemma}\label{G-neral} Assume that $G$ is generated by $F^*$, as in Lemma \ref{F^*}. Let $C_G<A^0$ be the connected Lie subgroup generated by the $C_g's,g\in F^*$, and let $a_G:A\to B_G:=A/C_G$ be the quotient map. 

1. $C_G$ is preserved by the action of $G$, and $G$ acts trivially on $B_G$. 

2. $a_G:A\to A/G$ induces a fibration $q_G:A/G\to B_G$ which is analytically locally trivial with fibre $C_G/G$.

3. $\nu_1^*(A/G)=q'(A/G)-n=-dim(C_G)$, and:

4. $dim(B_G)=q(A/G)=q'(A/G)=dim(B_G)$.

5. $q_G:A/G\to B_G$ is the Albanese map of $A/G$.

6. If $C_G/G$ is not uniruled, there exists a finite \'etale cover $\beta: B'_G\to B_G$ such that if $(A/G)':=(A/G)\times _{B_G}B'_G$, then $(A/G)'=(C_G/G)\times B'_G$.
\end{lemma}

\subsection{Proof of Lemma \ref{primary}.}\label{Sprimary}

We denote by $A^0:=Aut^0(A)$ the group of translations on $A$. The determination of the invariants of $(A/G)$ for $G$ cyclic, generated by $g$ of order $N$, rests on the following qualitative description of the action of $G$ on $A$: $(g-1):A\to A^0$ is nothing but the affine map sending $x\in A$ to the translation from $x$ to $g(x)$. Its image $Im(g-1)\subset A^0$ is thus a translate of a connected Lie subgroup $C_g$ of $A^0$. And $Im(g-1)=C_g$ if (and only if) it $Im(g-1)$ contains the neutral element of $A^0$, that is: if and only if $g$ has fixpoints on $A$. If $g$ has a fixpoint $a$, chosen as the neutral element of $A$, $Ker(g-1)$ is then a Lie subgroup of $(A,a)$, the elements of which are precisely the fixpoints of $g$ on $A$. The Lie subgroup $Ker(g-1)$ is thus well-defined, independently of the choice of the fixpoint $a$.

\begin{lemma}\label{g-1} Assume $g$ is of finite order $N>1$, and has a fixed point $a$, used as the zero of the addition on $A$.
 The fixed point set of $g$ is then $B_g:=Ker(g-1_A)$, and if $\widetilde{B_g}$ is its connected component containing the origin, it is a subtorus `complementary' to $C_g$, in the sense that $\widetilde{B_g} \cap C_g$ is a torsion subgroup of order dividing $N$, so that the addition map $s_g:C_g\times \widetilde{B_g}\to A$ is a $G$-equivariant isogeny, the action of $G$ on $C_g$, which is preserved by $G$, being induced by that on $A$, the action on $\widetilde{B_g}$ being trivial. \end{lemma}
 
 \begin{proof} In $\Bbb Z[X]$, we have Bezout equality: $N=U(X)-(X-1).V(X)$, with: $U(X)=1+X+\dots+X^{N-1},$ and: $V(X):=\frac{(U(X)-N)}{(X-1)}$. For any $b\in A$, $N.b=U(g).b-(g-1).(V(g).b),$ the first term is in $B_g$ (since $(g-1)U(g)=(g^N-1)=0$), the second is in $C_g:=Im(g-1)$, by definition. If $b\in B_g$, the second term vanishes; if $b\in C_g$, the first term vanishes, since $C_g$ is preserved by $g$, because $g.(g-1)a=(g-1).ga, \forall a\in A$. Thus, if $b\in \widetilde{B_g}\cap C_g$, both terms vanish, and $b$ is of $N$-torsion. \end{proof} 
 
 This establishes the first 3 claims, and hence the fourth claim, of Lemma \ref{primary}. We now turn to its fifth claim, when $N=p^r$ is primary. We may thus assume that $A=C_g$, and we only need to show that $\nu_1(A/G)=-\infty$, and $q'(A/G)=0$ if $g$ has an isolated fixpoint $a$. We thus assume this now. The action of $g$ on the tangent space to $A$ has then all of its eigenvalues different from $1$, and $g$ acts linearly on global coordinates $y_1,\dots,y_n$ on $A$, with an isolated singularity of type $\frac{1}{N}.(a_1,\dots ,a_n)$ and $0<a_i<N, \forall i=1,...,n$, which means that $g$ acts on these coordinates by $g.(y_1,\dots,y_n)=(u^{a_1}.y_1,\dots,u^{a_n}.y_n)$ where $u$ is any primitive $N$-th root of unity. Moreover, since $g$ has order $N=p^r$, not all of the $a_i's$ are divisible by $p$, and so at least one of them, say $a_1$, is prime to $p$. Replacing thus $g$ by $g^k$, where $k$ is an inverse of $a_1$ modulo $p^r$, we shall assume that $a_1=1$. The claims 5 and 7 of Lemma \ref{primary} then follow from the next Theorem \ref{Cyclic} and Lemma \ref{tpi_1}, while its Claim 6 follows from Claims 5 and 4.

\begin{theorem}\label{Cyclic} Assume that $G$ is cyclic, of order $N$, and has an isolated fixpoint $a$, at which $A/G$ has a singularity of type $\frac{1}{N}(a_1=1,a_2,\dots,a_n)$. Then:

 For any effective divisor $D$ on $A/G$, there exists $m(D)$ such that, if $m\geq m(D)$, $h^0(A/G, Sym^m(\Omega^1_{A/G})(D))=0$.

Thus: $\nu_1(A/G)=\kappa_1(A/G)=-\infty$. \end{theorem}

\begin{remark} When $n=1$, $g$ is a quasi-reflexion, and all the statements are easy (also proved below by removing the coordinates $y_i,i>1$).
\end{remark}
 
  \begin{proof}   Let $(y):=(y_1,\dots,y_n)$ be global linear coordinates on $A$ on which $g$ acts diagonally, as said above. We just need to consider these linear coordinates on a small neighborhood of $a$ in $A$. Let $\pi:(y_1,z_2,\dots,z_n):=(y_1,z)\to (y):=(y_1,y_2,\dots,y_n)$ be the (birational) map defined by $y_i:=z_i.y_{1}^{a_i}$ for $i=2,\dots, d$.  The map $\pi$ is $g$-equivariant if we let $g$ act on $(y_1,z)$ by: $g.(y_1,z):=(u.y_1,z)$, where $u$ is a primitive $N$-th root of unity, choosen as said above, before the statement of Theorem \ref{Cyclic}. Define the finite proper map  $q':(y_1,z)\to (y:=y_1^N,z)$. There is a unique birational map $p:(y,z)\to (\Bbb C^n/G)$ such that $p\circ q'=q\circ \pi:(y_1,z)\to (\Bbb C^n/G)$, with $q:\Bbb C^n\to \Bbb C^n/G$ the $G$-quotient. In fact, the map $p$ is nothing but a chart of the weighted blow-up $Z$ of the singularity $\frac{1}{N}(1,a_2,\dots,a_n)$, and so is a resolution of $\Bbb C^n/G$ at the generic point of the exceptional divisor over $a$. 
  
  We treat first\footnote{This is logically not necessary, but may help to follow the computations when $D\neq 0$.} the case when $D=0$.
 
 Let $w:=\sum_{(m)}c_{(m)}dy^{(m)}$be a `constant' section of $Sym^m(\Omega^1_{\Bbb C^n})$, that is: $(m):=(m_1,\dots,m_n)$ is an $n$-tuple of nonnegative integers of length $m:=\sum_im_i$ with $m>0$, $dy^{(m)}:=\otimes_idy_i^{\otimes{m_i}}$, and the $c_{(m)}\in \Bbb C$ are constants. Assume that there exists $v:=\sum_{(h)}b_{(h)}dy^{\otimes{h_1}}\otimes dz^{(h')}$ such that $q'^*(v)=\pi^*(w)$. Here $(h)=(h_1,h')$ is an $n$-tuple of nonnegative integers of length $m$, with $(h')=(h_2,\dots, h_n)$, and $b_{(h)}(y,z)$ are holomorphic functions in $(y,z)$. We shall prove that $c_{(m)}=0,$ for any $(m)$ of positive length $m$. This shows Theorem \ref{Cyclic} when $D=0$, since $(q')^*(H^0((A/G)',Sym^m(\Omega^1_{(A/G)'}))\subset \pi^*(H^0(A,Sym^m(\Omega^1_A)))$, if $(A/G)'$ is a smooth model of $A/G$.
 
 We will show this vanishing by comparing the coefficients of $(dy_1)^{\otimes m}$ in $\pi^*(w)$ and $q'^*(v)$. This coefficient in $q'^*(v)$ is $b_{(m,0,\dots,0)}.N^m.y_1^{(N-1).m}$. On the other hand, the coefficient of $dy_1^{\otimes m}$ in $\pi^*(dy^{(m)})$ is easily seen to be $(a.z)^{(m')}:=(a_2.z_2)^{m_2}.\dots.(a_n.z_n)^{m_n}.y_1^{(\sum(a_i-1).m_i)}$ for each $(m)=(m_1,m')=(m_1,m_2,\dots,m_n)$. The coefficient of $dy_1^{m}$ in $\pi^*(w)$ is thus: $$(\sum_{(m)}c_{(m)}.(a.z)^{(m')}.y_1^{(\sum(a_i-1)m_i)}).$$
 
 Notice that, since $a_i<N,\forall i=1,2,\dots,n$, and $a_1=1$, we have: $$\sum(a_i-1).m_i\leq (N-2).(\sum m_i)=(N-1).m-m.$$
 The equality $\pi^*(w)=q'^*(v)$ thus implies that: $$b_{(m,0,\dots,0)}(y,z).N^m.y_1^{(N-1).m}=\sum_{(m)}c_{(m)}.(a.z)^{(m')}.y_1^{(N-1).m-m-s((m))},$$
 where $s((m)):=(N-1).m-m-\sum(a_i-1).m_i\geq 0,\forall (m)$. Hence:
 
 $$b_{(m,0,\dots,0)}(y,z).N^m=\sum_{(m)}c_{(m)}.(a.z)^{(m')}.y_1^{-m-s((m))}.$$
 
 From which we conclude that $c_{(m)}=0,\forall (m)$, since $m>0$ and the $c_{(m)}$ are constants, the functions $z^{(m')}$ being linearly independent polynomials in the $z_i's$.
  The theorem is thus proved when $D=0$, with $m(D)=1$.
 
 We now prove the general case, when $D\neq 0$, by a completely similar computation. We may write a local equation\footnote{$A/G$ is $\Bbb Q$-factorial, so we may replace $D$ a suitable multiple to get $D$ Cartier.} of $D$ near $q'((\pi^{-1})(a))$ in the form: $y^{r}.F(y,z)$, $r$ being the multiplicity of the divisor $D$ at $a$, with $F$ not vanishing on the divisor $y_1^N=0$. We write a local section of $Sym^m(\Omega^1_Z)\otimes \mathcal{O}_Z(D)$ in the form: $v':=\frac{v}{y_1^{r.N}.F}$, where $v:=\sum_{(h)}b_{(h)}dy^{\otimes{h_1}}\otimes dz^{(h')}$, as before. The coefficient of $dy_1^m$ in $q'^*(v')$ is then: $b_{(m,0,\dots,0)}.N^m.y_1^{(N-1).m-r.N}.F^{-1}$. 
 
 If $q'^*(v')=\pi^*(w')$, for some $w'=\frac{w}{f(y_1,\dots,y_d)}$, where $w$ is as above and $f=f(y_1,\dots,y_d)$ is holomorphic nonzero, we get by the calculation above: $$(\sum_{(m)}c_{(m)}.(a.z)^{(m')}.y_1^{\sum(a_i-1).m_i})=y_1^{(N-1)m-r.N}.f.F^{-1}.b_{(m,0,\dots,0)}.N^m,$$ 

 or equivalently: $$\sum_{(m)}c_{(m)}.(a.z)^{(m')}.y_1^{(N-1).m-m-s((m))}=y_1^{(N-1)m-r.N}.f. F^{-1}.b_{(m,0)}.N^m$$ 
 
 $$=\sum_{(m)}c_{(m)}.(a.z)^{(m')}.y_1^{-s((m))}=y_1^{m-r.N}.f. F^{-1}.b_{(m,0)}.N^m.$$ 

We saw that $s((m))=(N-1).m-m-\sum(a_i-1).m_i\geq 0,\forall (m)$. So that $c_{(m)}=0,\forall (m)$, if $m>r.N$.\end{proof}

The next lemma concludes the proof of Lemma \ref{primary}.

\begin{lemma}\label{tpi_1} Let $G$ be a group of order $N>1$ acting on $A$. Assume that $G$ has an isolated fixed point $a\in A$. Let $\pi:A\to A/G$ be the natural quotient map. Then $\pi_*:\pi_1(A)\to \pi_1(A/G)$ is surjective, and $\pi_1(A/G)$ is finite, abelian, of exponent dividing $N$. In particular: $q'(A/G)=0$.
\end{lemma}

\begin{proof} (simplified version of my initial proof, suggested by one of the referees) The surjectivity of $\pi_*$ is clear, using $a$, and $\pi(a)$ as base points, since $\pi$ is totally ramified at $a$ (i.e: $\pi^{-1}((\pi(a)))=\{a\}$), so that loops at $\pi(a)$ lift to loops at $a$. We chose $a$ as the neutral element of the addition on $A$, and let $h$ be the linear map on $\Bbb C^n$ inducing $g$ on $A$. Since $a$ is an isolated fixed point, $1$ is not an eigenvalue of $h$, and so $1+h+\dots+h^{N-1}=0$. Thus $1+g_*+\dots+g_*^{N-1}=0$ as an endomorphism of $\pi_1(A)$. Let now $\pi_*:\pi_1(A)\to \pi_1(A/G)$ be the induced map by the natural quotient $\pi:A\to A/G$, the fundamental groups being based at $0\in \Bbb C^n$ and $a$. Since $\pi_*\circ g_*=\pi_*$, we get: $N.\pi_*=\pi_*\circ (1+g_*+\dots +g_*^{N-1})=\pi_*\circ 0=0$, and the second Claim.\end{proof}

\subsection{Proof of Lemma \ref{G-neral}.}

\

{\bf Notation:}\label{nb} For $A/G$, and $L$ a line bundle on $A/G$, we denote with $H^0(S^m(A/G)\otimes L):=H^0(X, Sym^m(\Omega^1_X)\otimes L_X)$ , if $X\to A/G$ is a resolution of the singularities, and $L_X$ the inverse image of $L$ on $X$, this vector space being independent of the resolution. If $B$ is a complex compact torus, $Sym^m(\Omega^1_B)$ denotes both the corresponding trivial vector bundle on $B$, and its space of sections. Write $\alpha_g: \widetilde{A/G}:=(C_g/G)\times \widetilde{B_g}\to A/G$ for the projection defined in the statement of Lemma \ref{primary}.

\begin{lemma}\label{B_g} In the situation of Lemma \ref{primary} ($G$ cyclic, generated by $g$ of primary order $p^r$, $B_g,C_g,\widetilde{B_g}$ defined in this Lemma), we have: 

1. $H^0(B_g,Sym^m(\Omega^1_{B_g}))=H^0(S^m(A/G)),\forall m\geq 0$.

2. If $L$ is a line bundle on $A/G$, there is a vector bundle $E_{B_g}$ on $B_g$, a line bundle $L_C$ on $C_g/G,$ and an integer $m(L)$ such that:

 $H^0(S^m(A/G)\otimes L))=\oplus_{0\leq j\leq m(L)}S_j\otimes Sym^{m-j}(\Omega^1_{B_g})\otimes H^0(B_g, E_{B_g})$ if $m\geq m(L)$, where $S_j:=H^0(S^j(C_g/G)\otimes L_C)$.
\end{lemma}

\begin{proof} {\bf Claim 1.} We have $\forall j>0$: 

$H^0(S^m((C_g/G)\times \widetilde{B_g}))=\oplus_{0\leq j\leq m} H^0(S^j(C_g/G))\otimes Sym^{m-j}(\Omega^1_{\widetilde{B_g}})$

$=Sym^m(\Omega^1_{\widetilde{B_g}})$, since $H^0(S^j(C_g/G))=\{0\}, \forall j>0$, for any $j>0$, by Theorem \ref{Cyclic}. 
On the other hand, $H^0(S^m( \widetilde{A/G}))=H^0(S^m(A/G)\otimes (\alpha_g)_*(\mathcal{O}_{\widetilde{A/G}}))$, since $\alpha_g:\widetilde{A/G}\to A/G$ is a Galois \'etale cover with abelian Galois group, and so $(\alpha_g)_*(\mathcal{O}_{\widetilde{A/G}})$ is a direct sum of line bundles, one factor being trivial, while the others are torsion non-trivial. The first part of the argument, applied after tensorisation with $(\alpha_g)_*(\mathcal{O}_{\widetilde{A/G}})$  then implies the claim. 

{\bf Claim 2.} The proof is similar, but uses in addition the fact that $q'(C_g/G)=0$, so that $Pic((C_g/G)\times \widetilde{B_g})= \gamma^*(Pic(C_g/G))+\beta^*(Pic(\widetilde{B_g}))$, up to torsion, and so $L=\gamma^*(L_C)+\beta^*(L_B)$ with $L_C,L_B$ line bundles on $(C_g/G), \widetilde{B_g}$ respectively, and $\gamma,\beta$ the obvious projections. Hence:

$H^0(S^m(\widetilde{A/G})\otimes L)$

$=\oplus_{0\leq j\leq m}H^0(S^j(C_g/G)\otimes L_C)\otimes H^0(\widetilde{B_g},Sym^{m-j}(\Omega^1_{\widetilde{B_g}})\otimes L_B)$

$=\oplus_{0\leq j\leq m(L_C)}H^0(S^j(C_g/G)\otimes L_C)\otimes H^0(\widetilde{B_g},Sym^{m-j}(\Omega^1_{\widetilde{B_g}})\otimes L_B),$ 

$=\oplus_{0\leq j\leq m(L_C)}S_j\otimes Sym^{m-j}(\Omega^1_{\widetilde{B_g}})\otimes H^0(\widetilde{B_g}, L_B),$  since there exists $m(L_C)$ such that $H^0(S^j(C_g/G)\otimes L_C)=\{0\},$ for all $  j\geq m(L_C)$, by Theorem \ref{Cyclic}, and since the vector bundles $Sym^{m-j}(\Omega^1_{\widetilde{B_g}})$ are trivial. 

In order to get Claim 2 (for $A/G$ instead of $\widetilde{A/G}$), we  just have to argue as in the proof of Claim 1, and to replace $H^0(\widetilde{B_g}, L_B)$ by $H^0(B_g, E_{B_g}),$ with the vector bundle $E_{B_g}:=(\alpha_g)_*(\mathcal{O}_{\widetilde{A/G}})$.\end{proof}

We now consider the case of a possibly non-cyclic $G$ generated by elements of primary orders having fixpoints.

For each $g\in F^*$, $C_g:=(g-1).A\subset A^0$ is a connected Lie subgroup, and thus so is $C_G$, the subgroup generated by the $C_g's$. Fix arbitrarily an origin $a$ in $A$ which identifies $A$ and $A^0$, and equips $A$ with a Lie group structure. The group $G=[g]$ generated by $g$ naturally acts on $C_g<A^0$ by: $h.(g-1).a':=(g-1).h.a'$ for any $h=g^k\in [g]$, and $a'\in A$. This is well-defined, since $(g-1).g^k.a'=0$ if $(g-1).a'=0$, since then $g^{k+1}.a'=g^k.a', \forall k\in \Bbb Z$.

We now prove the various claims of Lemma \ref{G-neral}.

\

{\bf Claim 1.} $C_G:=\sum_{g\in F^*}C_g\subset A$ is preserved by the action of $G$. Indeed, for $h\in G$, we have: 

$h.C_G:=h. (\sum_{F^*}(g-1)A)=\sum_{F^*}(h.g.h^{-1}-1).h(A)=C_G,$ 
since $F^*$ is stable by conjugation in $G$, and $h(A)=A$.

The quotient $a_G:A\to B_G:=A/C_G$ is thus well-defined, and $G$-equivariant for the trivial action of $G$ on $B_G$.  

{\bf Claim 2.} It thus induces a quotient map $q_G: A/G\to B_G$. For each $b\in B_G$, $q_G^{-1}(b)$ is naturally isomorphic to the quotient $C_G/G$. Indeed, by the equalities written above, for all $g\in F^*, a_g's \in A$, we have:

 $h.\sum_{g\in F^*}(g-1).a_g=\sum_{g\in F^*}(hgh^{-1}-1).h.a_g=\sum_{F^*}(g-1).(h.a_{h^{-1}.g.h}),$
 
 which shows that this action of $G$ is defined on $C_G<A^0$, and independent of the chosen origin $a\in A$.
 
 The fibration $q_G:A/G\to B$ is thus analytically locally trivial with fibre $C_G/G$.
 
 {\bf Claim 3.} We first prove that $\kappa_1^*(A/G)\leq dim(B_G)-n$. For any $g\in F^*$, recall that $C_g=(g-1).A<A^0$, and that $q'_g:A\to B_g:=A/C_g$ is the quotient.
 
 We have then two other quotients: $\pi_g:A/{[g]}\to A/G$, and $\beta_g:B_g\to B_G:=B_g/(C_G/C_g)$, such that $\beta_g\circ q_g=q_G\circ \pi_g:A/{[g]}\to B_G$.
 
 Denote, for $m\geq 0$, by $H^0(S^m(A/G))$ the vector space of sections of $Sym^m(\Omega^1_X)$ over $X$, if $X$ is any smooth model of $A/G$. Similarly for $A/{[g]}$.
 For any $m\geq 0$, we have a natural isomorphism $q_g^*:H^0(B_g,Sym^m(\Omega^1_{B_g}))\to H^0(S^m(A/{[g]}))$, and natural injective maps:
 
  $\beta_g^*:H^0(B_G,Sym^m(\Omega^1_{B_G}))\to H^0(B_g,Sym^m(\Omega^1_{B_g}))$, and:
  
   $\pi_g^*: H^0(S^m(A/G))\to H^0(S^m(A/{[g]})),$ hence an injective map:
   
    $(\pi_g\circ \beta_g)^*=(q_G\circ \beta_g)^*:H^0(S^m(A/G))\to H^0(B_g,Sym^m(\Omega^1_{B_g}))$.
 
 From this we deduce that:
 
  $H^0(S^m(A/G))\subset \cap_{g\in F^*}H^0(B_g,Sym^m(\Omega^1_{B_g}))\subset H^0(A,Sym^m(\Omega^1_{A}))$.
 
 The next Lemma \ref{cap} (given by lack of reference) shows that:
 $$\cap_{g\in F^*}H^0(B_g,Sym^m(\Omega^1_{B_g}))=Sym^m(\cap_{g\in F^*}H^0(B_g,\Omega^1_{B_g}))$$
 
 $=Sym^m(H^0(B_G,\Omega^1_{B_G})),$ since $B_G=A/C_G$, and $C_G=\sum_{g\in F^*}(C_g)$. 
 
 Hence $H^0(S^m(A/G))\subset H^0(B_G,Sym^m(\Omega^1_{B_G})), \forall m>0$. 
 
 Which implies the claimed inequality:  $$\kappa_1^*(A/G)\leq \kappa_1^*(B_G)-(dim(A/G)-dim(B_G))=0-(n-dim(B_G)).$$

 \begin{lemma}\label{cap} Let $F_j\subset E$ be a finite number of vector subspaces. For any $m\geq 0$, $Sym^m(\cap F_j)=\cap Sym^m(F_j)\subset Sym^m(E)$.
 \end{lemma} 
 
 \begin{proof} Induction on the number $k\geq 1$ of the $F_j's$. The case $k>2$ reduces by an easy induction to the case $k=2$, which we now consider, with $E=F_1+F_2$. In this case, write $E=F\oplus V_1\oplus V_2$, with $F:=F_1\cap F_2, F_j=V\oplus V_j$. Then, with $(h):=(h,h_1,h_2)$ running through triples of nonnegative integers of sum $m$, we have: 
 
 $Sym^m(E)=\oplus_{(h)} Sym^h(F)\otimes Sym^{h_1}(V_1)\otimes Sym^{h_2}(V_2)$, with
  $Sym^m(F_1)$ (resp. $Sym^m(F_2)$) defined by the conditions $h_2=0$ (resp. $h_1=0$). Their intersection is thus defined by $h_1=h_2=0$.
 \end{proof}

 \smallskip
 
 We now prove similarly that $\nu_1^*(A/G)=dim(B_G)-n$, too. Let $L$ be an effective line bundle on $A/G$, we denote by $L_g, L_{B_g}, L_G, L_A$ its inverse images on $A/{[g]}, B_g,B_G, A$ by means of the maps $\pi_g, \pi_g\circ q_g, q_G, q'_g\circ q_g\circ \pi_g$, if $q'_g:A\to B_g=A/C_g$ is the quotient.
 
 By the previous argument, and Lemma \ref{B_g}, we get, for each $g\in F^*$, an injective map $H^0(S^m(A/G)\otimes L)\to H^0(S^m(A/{[g]}\otimes L_g), \forall m>0$,
and, if $m\geq m(L,g)$ defined in Lemma \ref{B_g}, and $S_{g,j}:=H^0(S^j(C_g/{[g]})\otimes L_{C_g})$,an equality:

$H^0(S^m(A/{[g]})\otimes L_g))=[\oplus_{0\leq j\leq m(L,g)}S_{g,j}\otimes Sym^{m-j}(\Omega^1_{B_g})]\otimes H^0(B_g, L_{B_g}).$

For any $g,j\leq m, m>0$, define: $T_{g,j}:=S_{g,j}\otimes H^0(B_g, L_{B_g})$, we thus write the preceding equality as:

$H^0(S^m(A/{[g]})\otimes L_g))=\oplus_{0\leq j\leq m(L,g))}T_{g,j}\otimes Sym^{m-j}(\Omega^1_{B_g}),$ so that we get for any $g\in F^*,m\geq m(L,g)$ an inclusion:

$H^0(S^m(A/G)\otimes L_g))\subset \oplus_{0\leq j\leq m(L(g,j))}T_{g,j}\otimes Sym^{m-j}(\Omega^1_{B_g}).$

If we define $m(L):=sup_{g\in F^*} m(L,g)$, and $T_L:=\oplus_{\{j\leq m(L),g\in F^*\}}T_{g,j}$, we get an inclusion, for $m>0$:

$H^0(S^m(A/{[g]}\otimes L_{g}))\subset \oplus_{0\leq j\leq m(L))}T_L\otimes (\cap_{g\in F^*}Sym^{m-j}(\Omega^1_{B_g})),$ and finally by Lemma \ref{cap}:

$H^0(S^m(A/{[g]})\otimes L_{g}))\subset \oplus_{0\leq j\leq m(L)}T_L\otimes (Sym^{m-j}(\Omega^1_{B_G})),$ seen as a vector subspace of  $\oplus_{0\leq j\leq m(L))}T_L\otimes (Sym^{m-j}(\Omega^1_{A})),$
from which we get: $h^0(S^m(A/G)\otimes L)\leq m(L).dim(T).h^0(Sym^m(\Omega^1_{B_G}))$, for $m>0$.
  
This shows that $\nu_1^*(A/G)\leq dim(B_G)-n$, as claimed.

 {\bf Claim 4.} Thus: $\nu_1^*(A/G)\leq dim(B_G)-n\leq q(A/G)-n\leq q'(A/G)-n$. By Remark \ref{rem'q'}, $\nu_1^*(A/G)\geq \kappa_1^*(A/G)\geq q'(A/G)-n$, we thus have have $q(A/G)=dim(B_G)=q'(A/G)$, and $\kappa_1^*(A/G)=\nu_1^*(A/G)=q'(A/G)-n$.

{\bf Claim 5.} $q(A/G)=dim(B_G)$, so $q_G$ is the Albanese map of $A/G$. 

{\bf Claim 6.} This follows from \cite{Ca20},Theorem 3.4, since $q(C_G/G)=0$, $C_G/G$ is not uniruled, and so $Aut(C_G/G)$ is discrete since $q(C_g/G)=0$.

 \subsection{Decomposition of torus quotients.}\label{deco}
 
 Let $\pi:A\to A/G$ be the quotient of an Abelian variety\footnote{The results of this subsection certainly hold true for compact complex tori, but further arguments are then required.} $A$ by a finite group $G$ of affine automorphisms. Write $G=E\cup F$, with the same meaning as in \S\ref{groupaction} and Lemma \ref{F^*}. Let $F'$ be the normal subgroup of $G$, and $Q:=G/F'$ its quotient group. We have natural quotient maps: $\pi':A\to A/F'$ and $\pi_Q:A/F'\to A/G=(A/F')/Q$, the latter one being \'etale since the action of $Q$ on $A/F'$ is fixpoint free. Because of this, we shall assume that $G=F'$, since \'etale covers do not modify the invariants and fibrations we are considering here. 

 In this subsection we decompose $A/G$ into simpler varieties of the same form by means of its MRC fibration, followed by the Albanese map (well-defined for any K\"ahler $X$ with rational singularities).

 \begin{proposition}\label{propdeco} Assume that $G$ is generated by its elements having fixpoints. Let $a_G:A/G\to A_G$ be its (surjective and connected) Albanese map. Then $a_G=b_G\circ r_G$, where $r_G:A/G\to B/H$ is the MRC fibration of $A/G$, and $b_G:B/H\to A_G$ is the Albanese map of $B/H$. 
 
 Here, $r:A\to B=A/K$ is the quotient of $A$ by a subtorus $K$, preserved by a normal subgroup $S$ of $G$ such that $H=G/S$ acts effectively on $B$, with rationally connected quotient $K/S$. The fibration $r_G:A/G\to B/H$ is then a locally trivial fibre bundle with fibre $K/S$. 
 
 Next, $b_G$ is a locally trivial fibre bundle with fibre $L/H$, where $L$ is a subtorus of $B$ preserved by $H$ such that $A_G=B/L$. The fibre $L/H$ has only canonical singularities, and a torsion $\Bbb Q$-Cartier canonical bundle\footnote{ The torsion index is bounded by the largest $N$ such that $\varphi (N)$ divides $n$, $\varphi$ being the Euler totient function, $n:=dim(A)$.}.

 \end{proposition}

 We start the proof of Proposition \ref{propdeco}. First, recall the criteria for non-uniruledness of $A/G$  (see \cite{KL} for example):
 
 \begin{proposition}\label{pun} The following are equivalent, for a smooth model $X$ of $A/G$
 
 1. $X$ is not uniruled
 
  2. $\kappa(X)=0$
  
  3. The singularities of $A/G$ are canonical \footnote{This can be checked using the Reid-Tai criterion.}.
   
   \end{proposition}
  
   The structure of the `rational quotient' (or MRC fibration) is then easy to describe (see \cite{KL} for a different proof in a more general situation):

  \begin{lemma}\label{rqab} Let $A/G$ be as in \ref{pun}. There exists a subtorus $K<A$ with quotient $q:A\to B:=A/K$ and a normal subgroup $S<G$ preserving $K$, $H:=G/S$ thus naturally acting on $B$ with the following properties:
  
  The quotient $h:B\to B/H$ induces a quotient $r:A/G\to B/H$ such that $r\circ \pi=h\circ q: A\to B/H$ is the $MRC$-fibration of $A/G$. 
  
  Its (rationally connected) fibres are isomorphic to $K/S$, its base $B/H$ has canonical singularities and torsion canonical bundle.  \end{lemma}

  \begin{remark}\label{remalmhol} The proof in fact shows that the statements are true (except the last sentence of \ref{rqab}) if $f:A/G\dasharrow Z$ is any almost holomorphic fibration, not just the MRC fibration of $A/G$.
  \end{remark} 
  
  \begin{proof} Let $r_0:A/G\dasharrow R_0$ be any birational model of the MRC\footnote{It is easy to see that the MRC is defined for normal varieties, coincides with the one of a smooth model, and is almost holomorphic.} of $A/G$. It is almost holomorphic, that is: its generic fibre does not meet its indeterminacy locus. We choose for $r$ the `fibre-model' of $r_0$, defined as follows: for $z\in R_0$ generic, let $Y_z$ be the reduced fibre of $r_0$ over $z$. This defines a meromorphic map $\varphi:R_0\dasharrow Chow(A/G)$. We let $R$ be the normalisation of its image, and denote by $Y^+\subset R\times (A/G)$ the incidence graph of the family of cycles parametrised by $R$. The natural projection $d^+:Y^+\to (A/G)$ is thus birational, and $r:=p\circ (d^+)^{-1}:(A/G)\dasharrow R$ is the model of the MRC of $(A/G)$ we shall consider, with $p:Y^+\to R$ the first projection. Let $r\circ \pi: A\dasharrow R$ be the composed map. It is almost holomorphic. Let $r\circ \pi=s\circ q$ be the Stein factorisation, with $q:A\dasharrow B$ connected (i.e: with connected fibres), and $s:B\to B/H$ finite. From Lemma \ref{ah} below, we deduce that $q$ is holomorphic, and that $B=A/K$ is a quotient Abelian variety of $A$. The rest of the proof about $H,S$ is then quite standard, since the fibres of $r\circ \pi$ are exchanged by the affine action of $G$ on $A$. The last sentence follows from the fact that $B/H$ is not uniruled (by \cite{GHS}), it has thus only canonical singularities, by Proposition \ref{pun}, and a torsion canonical bundle (since it is a quotient $B/H$).\end{proof}

\begin{lemma}\label{ah} Let $q: A\dasharrow Z$ be an almost holomorphic map with connected fibres. The `fibre-model' of $q$ is then the quotient map from $A$ onto a quotient Abelian variety $B=A/K$.
\end{lemma}

\begin{proof}Since $q$ is almost-holomorphic, the standard `rigidity lemma' shows that the image $B^+$ of the associated fibre-map $\psi:Z\to Chow(A)$ is an irreducible component of $Chow(A)$. Let $A_z$ be a generic fibre of $q$ not meeting the indeterminacy locus of $q$. Then $B^+$ contains all of the translates of $A_z$ inside $A$. Since these translates cover $A$, and only one of them passes through the generic point of $A$, the family $B^+$ coincides with the family of translates of $A_z$, and also with the quotient of $A$ by the group $K$ of translations of $A$ preserving $A_z$. Thus, $A_z$ must be a single orbit of $K$, and is so a translate of a subtorus of $A$. 
\end{proof}

\begin{lemma}\label{factorisation} Assume $G=F'$. Let $r:A/G\to B/H$ be the MRC of $A/G$ described in Lemma \ref{rqab}, and let $a_G:A/G\to A_G$ be the Albanese map of $A/G$. Recall that $A_G=A/C_G$, and that $a_G$ is a locally trivial fibre bundle with fibre $C_G/G$, with the notations of Lemma \ref{G-neral}. Then $a_G=b_G\circ r_G$, where $b_G:B/H\to A_G$ is the Albanese map of $B/H$.
\end{lemma}

\begin{proof}The fibres of $r_G$ are rationally connected and thus mapped to points by $a_G$, there thus exists a factorisation $a_G=b_G\circ r_G$ with $b_G:B/H\to A_G$ a fibration. Necessarily $b_G$ is the Albanese map of $B/H$ since $a_G$ is the one of $A/G$.\end{proof}

 The quotients $A/G$ which are either rationally connected, or with canonical singularities (and so torsion canonical bundle) and finite fundamental groups are thus the two `building blocks' (the tori being the obvious third class) of this class of varieties. Their structure is not, by far, well-understood. Hence some remarks and questions about them.

 \begin{remark} If $A/G$ has a non-canonical singuarity and $A$ is simple, then $A/G$ is rationally connected. By \cite{KL}, Theorem 9, this can occur only when $dim(A)\leq 3$.
 \end{remark} 
 
\subsection{Examples, Questions}\label{r}  

\begin{question}

 1. If $A/G$ is rationally connected, is it  unirational? Is it rational? Is it uniformly rational in the sense of Gromov? See \cite{KL} for a thorough study of pairs $(A,G)$ such that $A/G$ is rationally connected. In particular, is the quotient $E^{\oplus 3}/\Bbb Z_4$ uniformly rational, where $E$ is the Gauss elliptic curve $\Bbb C/\Bbb Z[\sqrt{-1}]$, and $\sqrt{-1}$ acts diagonally on the three factors?
 
  2. For $n$ fixed, what are the possible finite groups $\pi_1(A/G)$ when $A/G$ is not uniruled, and $q'(A/G)=0$? 
 
 3. For $n$ fixed, what are the possible groups $G/F'$, Galois group of the \'etale cover $A/F'\to A/G$? 
 
 The groups $G$ acting freely (i.e with $F'=\{1\}$) on Abelian threefolds are classified in \cite{OS}, they are solvable and belong to a very short list.
 
 4. Classical examples of pairs $(A,G)$, especially when $n=2$, are described in \cite{BiL}, \S13. Poincar\'e reducibility shows that only the actions on products of simple Abelian varieties give indecomposable examples.

 5. Assume that $A/G$ has canonical singularities and trivial fundamental group. When does $A/G$ admit a crepant resolution (i.e: a resolution $X$ with $K_X$ trivial)? This is always the case when $n=2$, common when $n=3$, but rare when $n\geq 4$. The cases of holonomy $SU$ and $Sp$ behave differently.
 
 \end{question}
 
 \begin{remark}
  
 If such a crepant resolution $X$ does exist, the holonomy of a Ricci-flat K\"ahler metric on $X$ decomposes with irreducible factors of type either $SU$ or $Sp$. One can get information on this decomposition from the values of the dimensions of $H^{0,p}(A/G)=(H^{0,p}(A))^G$ and $\chi(\mathcal{O}_{A/G})$. The local existence of crepant resolutions is proved for all $Sl(3)$ quotients when $n=3$ (see the references in \cite{SS}).

In particular, any threefold $A/G$ with isolated canonical singularities and trivial canonical bundle is simply-connected, and has  crepant resolutions $X$ which thus admit Ricci-flat K\"ahler metrics of holonomy $SU(3)$. The simplest example is $A/G$, if $A=E^{\oplus 3}$, where $E$ is the elliptic curve with complex multiplication $\mu$ by a primitive cubic root of unity in $\Bbb C^*$, and $G$ the cyclic group of order $3$ acting on $A$ by $\mu$ simultaneously on each of the factors. There is only one other threefold $A/G$ with crepant resolution with holonomy $SU(3)$, by \cite{O}: it is the quotient of the Jacobian $J=E^3$ of the Klein quartic $C$ by a (cyclic) $7$-Sylow subgroup of $Aut(C)$, where $E$ is an elliptic curve with complex multiplication by $\frac{1+\sqrt{-7}}{2}$ (\cite{Prap}). 

Examples of quotients $A/G$ with crepant resolutions and irreducible holonomy $SU(n)$ in dimension $n\geq 4$ are given in \cite{CH}, \S 2 by a Kummer construction (see \cite{CH}, Corollary 2.3). These examples are simply-connected, by Lemma \ref{lch} below. The case when $G$ acts freely in codimension $2$ on $A$ is studied in the just posted preprint \cite{G'}, in which the non-existence is shown in dimensions $4,5$, and the lack of such examples in dimensions larger is noticed. Examples with irreducible $Sp$ holonomy are well-known in each even dimension (see \ref{exsp} below).\end{remark}

\begin{example}\label{exsp}

 1. Any finite group $G$ can be embedded as a transitive subgroup of some symmetric group $S_{n+1}$ on $n+1$ letters, and so act effectively on suitable Abelian varieties, as follows (\cite{beauv}). Let $B$ be any Abelian surface. Then $S_{n+1}$, and so $G$, acts on $B^{n+1}$ by permutation of the factors, preserving the Abelian subvariety $A:=\{(b_0\dots b_n)\vert b_0+\dots+b_n=0\}$. There is then a natural finite quotient map $K_{(G)}:=A/G\to A/S_{n+1}:=K_{(n)}$. By \cite{beauv}, $K_{(n)}$ is simply connected, has canonical singularities, trivial canonical bundle, and admits a crepant resolution $K^{[n]}$ which is an irreducible hyperk\"ahler manifold of dimension $2n$.  Thus $A/G$ has canonical singularities and trivial canonical bundle. Moreover, since any $1\neq g\in S_{n+1}$ fixes the points $(b,\dots,b)\in A$, if $b\in B$ is a point of $(n+1)$-torsion, we deduce from Lemma \ref{tpi_1}, that $\pi_1(A/G)$ is finite and abelian.
Indeed, it is easy to check that $(H^{0,1}(A))^G=\{0\}$, so $q(A/G)=q'(A/G)=0$, and Proposition \ref{G-neral} shows the finiteness of $\pi_1(A/G)$, which is abelian since $\pi:A\to A/G$ has total ramification at the previous torsion points, which implies the surjectivity of $\pi_*$ at the level of fundamental groups. Moreover, $(H^{0,2}(A))^G=\Bbb C$ in this situation.
 
 2. In the preceding construction, if $G$ is embedded in $S_n$, let $S_n$ act on $B^n$ by permutation of the factors. This action commutes with the action of $S_2^n$ by multiplication by either $+1$ or $- 1$ on each of the factors, and induces a Galois quotient $B^n\to T^n/S_n$, where $T=B/S_2$ is the Kummer surface of $B$. Taking the Hilbert scheme crepant resolution and deformations of $T^n/S_n$ leads to the second classical family of Hyperk\"ahler manifolds found in \cite{beauv}. One can restrict the action of $S_n$ on $T^n$ to $G$ in order to get a class a varieties having the same properties as in the preceding example $K_G$.
 
 A different quotient $A/G$ is given in \cite{BDW}, which constructs a crepant resolution $X$ of a quotient $E^4/G$, with $E$ the Gauss elliptic curve and $G$ a group of order $32$. However, $X$ turns out to be a deformation of the Hilbert square of $K3$ surfaces. 
 
 3. Numerous examples of quotients $A/G$ arise when $A$ is the Jacobian variety of a curve $C$ of genus $g>1$ with a non-trivial group $G$ of automorphisms. See \cite{OS} and \cite{Prap} in dimension $3$. Examples of uniruled $A/G's$  are given in \cite{Beauv"} for $g\leq 4$, where it is additionally shown that if the genus $g$ of $C$ is at least $21$, then $A/G$ is not uniruled (so that $\kappa(A/G)=0$). It is moreover shown in \cite{Beauv"} that $A/G$ is not uniruled if $g=5$, and the author suspects that $A/G$ is not uniruled if $g\geq 6$.

\end{example}

\begin{lemma}\label{lch} Let $G$ be a finite group acting on the complex manifolds $X_i,i=1,2$, and diagonally on the product $X:=X_1\times X_2$. Assume that $G$ fixes the points $a_i\in X_i,i=1,2$. Let $Y_i:=X_i/G,i=1,2$. The natural quotient map $q:Y:=X/G\to Y_1\times Y_2$ induces a morphism of groups $q_*:\pi_1(Y)\to \pi_1(Y_1\times Y_2)$ which is isomorphic.

In particular, $Y$ is simply-connected if so are $Y_i,i=1,2$.
\end{lemma}

\begin{proof} Let $r:X\to Y$ be the quotient map. Since $a:=(a_1,a_2)$ is fixed by $G$, the map $r_*:\pi_1(X)\to \pi_1(Y)$ is surjective (the base-points being $a$ and $r(a)$). Since $\pi_1(X)$ is the direct product of the $\pi_1(X_i)$, $\pi_1(Y)$ is generated by its subgroups $H_1:=r_*(\pi_1(X_1\times \{a_2\}))$ and $H_2:=r_*(\pi_1(\{a_1\}\times X_2))$.These subgroups are isomorphic to $\pi_1(Y_1)$ and $\pi_1(Y_2)$ respectively, since the $q_2$ and $q_1$-fibres of $Y$ over $b_2:=q_2\circ r (a_2)$ and $b_1:=q_1\circ r(a_1)$ respectively are isomorphic to $Y_1$, $Y_2$. Here $q_i:Y\to X_i,i=1,2$ are the compositions of $q:Y\to Y_1\times Y_2$ with the projections on the factors. We thus obtain a group morphism: $H_1\times H_2\to \pi_1(Y)$ which is an inverse to $q_*$.
\end{proof}

\section{Remarks on Special Manifolds.}\label{Sspec}

For more details on the orbifold notions, see  \cite{Ca04} and \cite{Ca11}.

We briefly recall the (conditional in the conjecture $C_{n,m}^{orb}$ of \cite{Ca04}) decomposition of the core map $c=(J\circ r)^n, n=dim(X)$ described in this text, 
by means of iterated fibrations $J\circ r$, the fibrations $r,J$ being defined as follows, for any smooth geometric orbifold $(X,\Delta)$.

$\bullet$ $r:(X,\Delta)\to (R,\Delta_r)$ is the birationally unique fibration such that its orbifold base $(R,\Delta_r)$ has $\kappa\geq 0$, and its generic (smooth) orbifold fibres $(X_r,\Delta_{\vert X_r})$ have $\kappa^+=-\infty$.
The meaning of $\kappa^+(Y,\Delta_Y)=-\infty$ is that for any rational fibration $f:(Y,\Delta_Y)\dasharrow (Z,\Delta_f)$, one has: $\kappa(Z,\Delta_f)=-\infty$. The conjecture $C_{n,m}^{orb}$ implies the existence of $r$.

When $\Delta=0$, $r: X=(X,0)\to (R,\Delta_r)$ is nothing else but the usual `rational quotient' (or MRC fibration) of $X$, with $\Delta_r=0$ by \cite{GHS}.

$\bullet$ $J:(X,\Delta)\to (J,\Delta_J)$ is the Moishezon-Iitaka fibration, defined when $\kappa(X,K_X+\Delta)\geq 0$, by sections of suitably divisible multiples of $(K_X+\Delta)$. Its generic orbifold fibres have $\kappa=0$, and $dim(J)=\kappa(X,K_X+\Delta)$.

We may assume that $r$ and $J$ are regular, replacing $(X,\Delta)$ by a suitable orbifold birational model $(X',\Delta')$ (obtained by putting large enough multiplicities on the exceptional divisors of the modification $X'\to X$, and the original ones on the strict transform of $\Delta)$. Such an orbifold modification does not change the sections of the orbifold cotangent tensors). It coincides with the usual Moishezon-Iitaka fibration when $\Delta=0$.

Then $c:=(J\circ r)^n:(X,\Delta)\to (C, \Delta_c)$ is the `core map' of any $n$-dimensional $(X,\Delta)$, with its orbifold base $(C,\Delta_c)$ of general type, and orbifold fibres special. In particular, $C$ is a point if and only if $(X,\Delta)$ is special.

The relative versions of the core map and its decomposition are also true: if $f:X\to Y$ is any fibration (with $X,Y$ compact K\"ahler smooth), the maps $J_f,r_f$ are also defined over the base $Y$, inducing $J,r$ on the general smooth fibres of $f$, and give the core decomposition of the general fibre of $f$. Recall that the existence of the maps $r$ is conditional in the Conjecture $C_{n,m}^{orb}$.

We shall now show how Conjecture \ref{Conjspec} claiming that $\nu_1^*(X)=\kappa_1^*(X)=q'(X)-n$ could be possibly proved\footnote{There are several quite serious difficulties to overcome, especially with the multiple fibres and orbifold divisors.}, when $X$ is a special compact K\"ahler $n$-fold. 

We may and shall assume that $q(X)=q'(X)$, by replacing $X$ with a suitable finite \'etale cover, which is still special, by \cite{Ca04}. The Albanese map $a_X:X\to A_X$ is a fibration with special smooth fibres, after \cite{CC}, and we can apply to it the relative $J\circ r$ decomposition.

Assume\footnote{In order to see what already happens in the much simpler case of absence of multiple fibres.} that (possibly after a finite \'etale cover), the orbifold fibres of the fibrations $J,r$ in the relative decomposition sequence $c=(J\circ r)^n$ of the core map of $a_X$  all have zero orbifold divisors\footnote{This might follow from the `Abelianity conjecture', claiming that $\pi_1(X)$ is virtually abelian if $X$ is special.}. The manifold $X$ thus appears as a tower of fibrations with fibres either rationally connected, or with $\kappa=0$. The latter ones conjecturally admit good minimal models, and thus fibre over their Albanese varieties with fibres having $\nu_1=-\infty$. Conditionally in this good minimal model conjecture, we thus inductively only need to prove that if $f:Y\to Z$ is a term of this decomposition,  with trivial orbifold base, and with smooth fibres either rationally connected, or with $\kappa=0$ and $\nu_1=-\infty$, or complex tori, then $\nu_1(Y)+dim(Y)=\nu_1(Z)+dim(Z)$ since $q'(X)=q'(Y)=q'(Z)$.

When the smooth fibres of $f$ are rationally connected (resp. have $\nu_1=-\infty$), this follows from Theorem \ref{tRQ} (resp. Corollary \ref{cRQ}, since the orbifold divisor of the orbifold base is assumed to be zero, the hypothesis 3 of this Corollary is satisfied). The (highly) non-trivial case is when $f:Y\to Z$ has smooth fibres complex tori. Even when $dim(Y)=2,dim(Z)=1$, some cases are still open (see \cite{HP20} and \S\ref{Ssurf} below).

We raise this explicitly as a possibly tractable problem.

\begin{question}\label{qalb} Let $f:Y\to Z$ be a fibration with $Y,Z$ smooth projective. Assume that $q'(Y)=q'(Z)$, that the smooth fibres of $f$ are Abelian varieties, and that $f$ has no multiple fibre in codimension one. Do we then have: $\nu_1^*(Y)+dim(Y)=\nu_1^*(Z)+dim(Z)$? Is this true at least in the two extreme cases when $f$ is either isotrivial, or has maximal variation?
One can probably reduce the general question to these two extreme cases in general.
\end{question}

The computation of $\kappa_1^*$ might be much easier (replacing $\nu_1$ by $\kappa_1$ in the preceding question) when the variation of $f$ is maximal. When fibre and base are one-dimensional, this is true as a consequence of the next proposition \ref{exnu>k}, proved below and already stated in Example \ref{nu>k}. This is due to Atiyah when $X$ is an elliptic curve (used in \cite{Sak}, Example 3), by a similar use of the Clebsch-Gordan formula. 

\begin{proposition}\label{exnu>k} Let $0\to \mathcal{O}_X\to E\to \mathcal{O}_X\to 0$ be a non-split extension, and let $A$ be an ample line bundle such that $h^1(X,A)=0$ on $X$, smooth projective. Then:

1. $h^0(X,Sym^m(E))=1,\forall m\geq 0$.

2. $h^0(X,Sym^m(E)\otimes A)=(m+1).h^0(X,A),\forall m\geq 0$.

In particular: $\kappa(X,E)=-1<\nu(X,E)=0$.
\end{proposition}

\begin{proof} Everything follows easily from (a small part of the special case when $det(E)=0$) of the Clebsch-Gordan formula for $m\geq 1$: 

$Sym^m(E)\otimes E=Sym^{m+1}(E)\oplus Sym^{m-1}(E)$, with $Sym^0(E)=\mathcal{O}_X$.

Claim 1. We have an exact sequence:

$0\to Sym^m(E)\to Sym^m(E)\otimes E\to Sym^m(E)\to 0,$ 

which implies that $h^0(X, Sym^m(E))\otimes E)\leq 2.h^0(Sym^m(E))$, and so

$2.h^0(X, Sym^m(E))\geq h^0(X, Sym^{m+1}(E))+h^0(X, Sym^{m-1}(E))$.

Since $h^0(X,Sym^m(E))=1$ for $m=0,1$, we get inductively:

$h^0(X,Sym^m(E))= 1$. 

Claim 2. The proof is entirely similar, tensoring the exact sequences with $A$, and using the fact that $h^1(X,A)=0$. The induction step includes however, not only that $h^0(X,Sym^m(E)\otimes A)\leq (m+1).h^0(X,A)$, but also that $h^1(X,Sym^m(E)\otimes A)=0,\forall m\geq 0$. The easy details are left to the reader.
\end{proof}

\section{Remarks on Surfaces not of general type.}\label{Ssurf}

We consider here smooth compact K\"ahler surfaces not of general type (see Remark \ref{rem ample} for simple remarks on these). They have already been thoroughly studied in \cite{HP20} from the same viewpoint as here, with which the present section overlaps considerably (and partially rests). Recall (\cite{Ca04}) that a compact K\"ahler surface $X$ not of general type is `special' if and only if no finite \'etale cover of $X$ fibres over a curve of genus at least $2$, or equivalently if $q'(X)\leq 2$.

Recall the statement:

\begin{proposition}\label{psurf'} If $X$ is a smooth projective surface not of general type, then $\nu_1(X)=\kappa_1(X)$, unless possibly when $\kappa(X)=1$, the elliptic fibration is not isotrivial, and $q'(X)>0$. 

Moreover, $\kappa_1(X)=q'(X)-2$ if $X$ is special. 

If $X$ is not special,  replacing $X$by a suitable finite \'etale cover\footnote{Except in the easy case when $f:X\to \Bbb P^1$ has two multiple fibres of different multiplicities.This particular case is easily treated by going to a cyclic (ramified) cover over the two multiple fibres, which fibres over $\Bbb P_1$ without multiple fibre.}, there is a fibration $f:X\to C$ over a curve $C$ of genus at least $2$ with smooth fibres $F$ either rational or elliptic. If $F$ is rational, $\kappa_1(X)=0$. If $F$ is elliptic, $\kappa_1(X)=0$ if $f$ is not isotrivial, and $\kappa_1(X)=1$ if $f$ is isotrivial.
\end{proposition} 

\begin{remark} The cases in which Conjecture \ref{ConjO} is left open are the following ones: $\kappa(X)=1$, and after a finite \'etale cover, the elliptic fibration $f:X\to C$ has no multiple fibre, and $g(C)\geq 1$. In these cases, $\kappa_1(X)=-1$ if $g(C)=1$, and $\kappa_1(X)=0$ if $g(C)\geq 2$.
\end{remark}

\begin{proof} When $\kappa(X)=-\infty$, $X$ is birationally a product $\Bbb P_1\times C$, thus $\nu_1^*(X)=\kappa_1^*(X)=\kappa^*(\Bbb P_1)+\kappa_1^*(C)=-1+min\{1,(g-1)\}$.

If $\kappa=0$, $X$ is birationally, after a finite \'etale cover, either a complex torus, or a $K3$ surface. In the first case, $\nu_1(X)=0$, in the second case, $\nu_1(X)=-\infty$ (by \cite{HP18}, for example).

The main case is thus when $\kappa(X)=1$. In this case, replacing $X$ by a suitable finite \'etale cover, the elliptic fibration $f:X\to C$ has no multiple fibre\footnote{Except in the easy case when $f:X\to \Bbb P^1$ has two multiple fibres of different multiplicities, treated as before.}.

We have different cases, according to the genus $g$ of $C$, and to whether or not $f$ is isotrivial.

$\bullet$ First case: $f$ is not isotrivial. 

We then have: $\kappa_1^*(X)= \kappa_1^*(C)-1=min\{1,(g-1)\}-1$. Thus $\kappa_1^*(X)=-2,-1,0$ if $g=0,1$, or $g\geq 2$ respectively. This follows from \cite{Sak}, Example 3, when $g(C)\geq 2$, but remains true if $g(C)\leq 1$. Indeed, since $\kappa(F, E)=-1$ when $F$ is an elliptic curve and $E$ is a non-split extension of two trivial line bundles over $F$ (see Proposition \ref{exnu>k}), we have $f_*(Sym^m(\Omega^1_X))=mK_C$ over the Zariski-open subset of $C$ over which $f$ is submersive. We thus have: $f_*(Sym^m(\Omega^1_X))=f_*(f^*(mK_C)^{sat})$, where $f^*(mK_C)^{sat}$ is the saturation of $f^*(mK_C)$ inside $Sym^m(\Omega^1_X)$. Since $f$ has no multiple fibre\footnote{We implicitely use the fact that `classical' and `non classical' multiplicities coincide for elliptic fibrations.} by assumption, the quotient $f^*(mK_C)^{sat}/f^*(mK_C)$ is  a divisor $E$ `partially supported on the fibres of $f$' (see Definition \ref{PSF}), and so one has, by Lemma \ref{lemmapsf} (of which only the easy case $dim(T)=2, dim(W)=1$ is used here): $f^*(mK_C)^{sat}\subset f^*(mK_C)\otimes \mathcal{O}_X(sE)$, for a certain positive integer $s$ independent of $m$. Thus $f_*(f^*(mKC)^{sat})\subset mK_C\otimes f_*(\mathcal{O}_X(sE))\subset (mK_C\otimes H)^{\oplus R}$, for a certain $H$ ample on $C$, and $R>0$, by Lemma \ref{inject}. This shows the claim since $\nu(K_C)=\kappa(K_C)$ for every curve $C$.

The remaining problem is the equality $\nu_1(X)=\kappa_1(X)$. 

If $g=0$, and if $\kappa_1=-\infty$, this is proved  in \cite{HP20}, Theorem 1.2, using analytic methods (see Proposition 5.4 in the non-isotrivial case).

We are thus left with the two cases $g=1,g\geq 2$.

If $g=1$, we must show that $\nu_1^*(X)=-1$, thus exclude the existence of an ample $A$ on $X$ such that $h^0(X, Sym^m(X)\otimes A)$ grows at least linearly with $m$. 

If $g\geq 2$, we must show that $\nu_1^*(X)=0$, thus exclude the existence of an ample $A$ on $X$ such that $h^0(X, Sym^m(X)\otimes A)$ grows at least quadratically with $m$. 

The difficulty is that, if $F$ is a smooth fibre of $f:X\to C$, then $E:=\Omega^1_{X\vert F}$ is a non-split extension of two trivial line bundles, and so (see proposition \ref{exnu>k}) $h^0(F,Sym^m(E)\otimes A)=(m+1).h^0(F,A), \forall m>0,$ $A$ ample on $F$, while $h^0(F,Sym^m(E))=1$. The equality $\nu_1(X)=\kappa_1(X)$ cannot thus be deduced (at least in an immediate way) from local arguments near the generic smooth fibre, as the ones given above for computing $\kappa_1$.

We leave these two cases open.

$\bullet$ Second case: $f$ is isotrivial. 

If $g(C)\geq 2$, it follows from \cite{Sak}, Example 3, that $\kappa_1(X)=1$. Remark \ref{rnu=k"} thus implies that $\nu_1(X)=\kappa_1(X)$ in this case. We shall recover this equality also when $g(C)=1$ by a more explicit method.

Let $F$ be a smooth (elliptic) fibre of $f:X\to C$.  When $C=\Bbb P_1$ and $F$ has no automorphism other than translations, the equality $\nu_1=\kappa_1$ is shown in \cite{HP20} by expressing $X$ as a global quotient $(F\times C')/G$. We shall prove this for any $F$, and compute $\kappa_1(X)=\nu_1(X)$ by a different method, analyzing the sections of $Sym^m(\Omega^1_X)$ near a singular fibre of $f$ (if any), by means of a computation similar to the one made to prove Theorem \ref{Cyclic}. We may indeed assume that $X$ is minimal, and has at least one singular fibre, since otherwise $X$ has a finite \'etale cover which is a product $F\times C'$. We are thus locally near such a singular fibre in the following situation.

\begin{lemma}\label{singfib} Assume that $f:X\to C$ is isotrivial with smooth fibres $F$, $X$ being minimal. Let $a\in C$ be a point over which $f:X\to C$ has a singular fibre. Let $a\in U\subset C$ be a small disc around $a$, and assume that $X_U:=f^{-1}(U)$ is a minimal resolution of $(U\times F)/G$, where $G$ is a cyclic group of order $N=2$ or $3$, generated by $g$, and acting on $U\times F$ by $g.(s,x)=(\zeta.s,\zeta^{\varepsilon}x)$, where $\zeta$ is a primitive $N$-th root of $1$, and $\varepsilon\in \{1,-1\}$. 

Let $w:=\sum_{p=0}^{p=m}c_p(s)(ds)^p.(dx)^{m-p}$ be a holomorphic $G$-invariant section of $Sym^m(\Omega^1_{U\times F})$. Then $w$ descends to a holomorphic section of $Sym^m(\Omega^1_{X_U})$ if and only if:

1. $w=\sum_{p=0}^{p=m} a_p(s^N)(d(s^N))^{m-p}.(s^{N-1}.dx)^{p}$ if $\varepsilon=1$,

2. $w=\sum_{p=0}^{p=m} a_p(s^N)(d(s^N))^{m-p}.(s.dx)^{p}$ if $\varepsilon=-1$, 

the $a_p's$ being holomorphic in both cases.

3.Thus, if $a_1,\dots, a_r$ are the points of $C$ over which $f$ has singular fibres, $Sym^m(\Omega^1_X)\subset f^*(\oplus_{p=0}^{p=m}(K_X^p)\otimes \mathcal{O}_C(-[\frac{m-p}{N}].(a_1+\dots+a_r)))$.\end{lemma}

\begin{proof} Assume $\varepsilon=1$, the second case is proved by a completely similar computation (just replacing $z=\frac{x}{s}$ by $z:=\frac{x}{s^{N-1}}$). Let $(s,x)=(0,0)$ be a fixed point of $g$ on the fibre $\{0\}\times F$. Blow-up this point. We have thus coordinates $(s,z:=\frac{x}{s})$ on this blow-up. 

Then: $w=\sum_pc_p(s).ds^{m-p}.(sdz+zds)^p$

$=\sum_pc_p(s).ds^{m-p}.(\sum_{h=0}^{h=p}(C_{p}^h.s^h.z^{p-h}ds^{p-h}dz^h))$

$=\sum_{h=0}^{h=m}s^h.ds^{m-h}dz^h.(\sum_{p=h}^{p=m}C_{p}^{h}c_p(s)z^{p-h})$.

On the other hand, in the coordinates $(s^N,z)$, $w$ takes the form:

$w=\sum_{h=0}^{h=m}b_h(s^N,z).d(s^N)^{m-h}.dz^h$, with $b_h$ holomorphic. 

Identifying both expressions, we get:

$(N-1)^{m-h}.b_h(s^N,z)s^{(N-1).(m-h)}=s^h.(\sum_{p=h}^{p=m}C_{p}^{h}c_p(s)z^{p-h}),\forall h$.

For $h=0$, this gives: $(N-1)^{m}.b_0(s^N,z)s^{(N-1).(m)}=(\sum_{p=0}^{p=m}c_p(s)z^{p})$, which implies that $c_p(s)=s^{(N-1).m}a_p(s^N), \forall p\geq 0$, $a_p$ holomorphic.

Thus $w=\sum_pc_p(s).ds^{m-p}.dx^p=\sum_pa_p(s^N).s^{(N-1).m}.ds^{m-p}.dx^p=\sum_pa_p(s^N).s^{(N-1).(m-p)}.ds^{m-p}.s^{(N-1).p}dx^p$, which is Claim 1.

Claim 3. The statement follows immediately from Claims 1 and 2 if all singular fibres of $f$ arise from a cyclic quotient of order $2$ or $3$. Since a cyclic group $G$ of automorphisms of any elliptic curve has order a divisor of $6$ or $4$, Claim 3 also applies to any isotrivial elliptic fibration $f:X\to C$ without multiple fibre, since we can factor a quotient $(U\times F)/G$  as $((U\times F)/H))/(G/H)$, with $H$ of order $N=2$, $G/H$ of order $3$ or $2$, and the sections of $Sym^m(\Omega^1_{X_U})$ for $X_U$ a a smooth model of $(U\times F)/G$ lift to those of a smooth model of $(U\times F)/H$ if $G$ has order $6$ or $4$. (We could also have worked directly with these quotients at the expense of some more direct computations.)\end{proof}

\begin{corollary}\label{surfk=1} Let $f:X\to C$ be an isotrivial elliptic fibration from a minimal surface $X$. Assume that $f$ has no multiple fibre, and at least one singular fibre. Then $\nu_1(X)=\kappa_1(X)=\kappa(C)-1$ if $g(C)\leq 1$, and $\nu_1(X)=\kappa_1(X)=\kappa(C)$ if $g(C)\geq 2$.
\end{corollary}

\begin{proof} The inequality $\kappa_1(X)\geq \kappa(C)-1$ is clear, since $f^*(mK_C)\subset Sym^m(\Omega^1_X)$ for each $m\geq 0$.  We have equality if $g(C)\leq 1$, but $\kappa_1(X)=\kappa(C)$ if $g(C)\geq 2$, by \cite{Sak}, Example 3, p. 553 . We shall thus prove that $\nu_1(X)\leq \kappa_1(X)$. Let $A$ be an ample line bundle on $X$. By Lemma \ref{inject}, there are an ample line bundle $H$ on $C$, a positive integer $R$, and an injection $f_*(A)\subset H^{\oplus R}$. From the preceding Lemma \ref{singfib}, for each integer $m\geq 0$, there is an injection:

$Sym^m(\Omega^1_X)\otimes A\subset \oplus_{p=0}^{p=m}f^*(K_X^{m-p}\otimes -[\frac{p}{N}]L\otimes H)^{\oplus R},$

where $L$ is a line bundle of degree $\ell$ at least $1$ on $C$, and $N=2$, or $3$. The degree of $K_X^{m-p}\otimes ( -[\frac{p}{N}]L)\otimes H)$ is $d(m,p)=(m-p)2(g-1)-[\frac{p}{N}].\ell+d$, if $d$ is the degree of $H$. Let $h(m):=h^0(X,Sym^m(\Omega^1_X)\otimes A)$. 

$\bullet$  If $g(C)=0$, $d(m,p)<0, \forall p, \forall m>N.d$. Thus $\nu_1(X)=-\infty$.
 
$\bullet$  If $g(C)=1$, $d(m,p)=-[\frac{p}{N}].\ell+d<0$ if $p>\frac{Nd}{\ell}$. Thus:
 
  $h(m)\leq R. (\sum_{p=0}^{p=Nd} (d-[\frac{p}{N}].\ell))\leq RNd^2$ is bounded, and $\nu_1(X)=-1$.
 
 $\bullet$ If $g(C)\geq 2$, then $h(m)\leq R. (\sum_{p=0}^{p=m}((2m(g-1))+d)$
 
 $=R.(2(g-1).m^2+m.d)$. Hence $\nu_1(X)=2-1=\kappa(C)=\kappa_1(X)$. This completes the proof of Proposition \ref{psurf'}. \end{proof}\end{proof}

\section{Pseudoeffectivity of subsheaves of $\Omega^p_X$.}\label{Sfol}

Pseudoeffective line bundles $L\subset \Omega^p_X$ have been intensively studied. Bogomolov proved that $\kappa(X,L)\leq p$. The strengthening $\nu(X,L)\leq p$ has been obtained by \cite{Mou}, \cite{Bou}. It is also proved in \cite{Dem} that the distribution annihilating $L$ is a foliation (however possibly of rank larger than $(n-p)$).

The equality $\nu=\kappa$ however fails in general for these sheaves (either subsheaves, or quotients of $\Omega^p_X$). The failure seems however to be linked with quite peculiar and interesting geometrical situations.

\begin{question}\label{rnu>kappa} The presently known examples of $L\subset \Omega^p_X$ of rank one, saturated, such that $\nu(X,L)>\kappa(X,L)$ have also $\kappa(X,L)=-\infty$. 

Are there examples with $\nu(X,L)>\kappa(X,L)\geq 0$?
\end{question}

{\bf Acknowledgements:} Many thanks to the referees for their careful reading of the first versions, pointing out gaps, inaccuracies at many places. And also for simplifications in the proofs of Lemmas \ref{lnofix} and \ref{tpi_1},  and also for pointing out the insufficiency of the proof of Lemma \ref{lemmapsf} in the first version of the text. Their help lead to a considerable improvement of the text.

\end{document}